\documentclass[a4paper]{amsart}

\usepackage{graphicx}
\usepackage{amsfonts,amsmath,amssymb, amscd, amsthm}
\usepackage[colorlinks=true]{hyperref}

\newcommand\pf{\begin{proof}}
\newcommand\epf{\end{proof}}

\newcommand\Alg{\mathcal{A}lg}
\newcommand\Gal{\operatorname{Gal}}
\newcommand\Hom{\operatorname{Hom}}
\newcommand\Spec{\operatorname{Spec}}
\newcommand\id{{\operatorname{id}}}
\renewcommand\H{\mathrm{H}}
\renewcommand\AA{\mathcal A}
\newcommand\OO{\mathcal O}
\newcommand\eps{\varepsilon}

\newcommand{\MG}{{}^G\mathcal{M}}
\newcommand{\MGb}{{}^G\mathcal{M}_b}
\newcommand{\bt}{\boldsymbol{\otimes}^b}
\renewcommand{\d}{\textnormal{d}}
\newcommand{\Z}{\mathbb{Z}}
\newcommand{\C}{\mathbb{C}}
\newcommand{\R}{\mathbb{R}}
\newcommand{\slt}{\mathfrak{sl}_2}

\newcommand{\B}{\mathcal{B}}
\newcommand{\U}{\mathcal{U}}

\newtheoremstyle{pedro}{}{}{\itshape}{}{\sc}{~--}{ }{\thmname{#1}\thmnumber{ #2}\thmnote{ (#3)}}

\newtheoremstyle{pedroex}{}{}{}{}{\sc}{~--}{ }{\thmname{#1}\thmnumber{ #2}\thmnote{ (#3)}}

\theoremstyle{pedro}
\newtheorem{lem}{Lemma}[section]

\newtheorem{thm}[lem]{Theorem}

\newtheorem{prop}[lem]{Proposition}

\newtheorem{coro}[lem]{Corollary}

\theoremstyle{remark}

\theoremstyle{pedroex}

\newtheorem{ex}[lem]{Example}

\newtheorem{Rem}[lem]{Remark}

\numberwithin{equation}{section}
\title[Twisting algebras -- explicit computations]
{Twisting algebras using non-commutative torsors -- Explicit computations}

\author{Pierre Guillot, Christian Kassel and Akira Masuoka}

\address{
(P.G.)
Universit\'{e} de Strasbourg \& CNRS,
Institut de Recherche Math\'{e}matique Avan\-c\'{e}e,
7~Rue Ren\'{e} Descartes,
67084 Strasbourg Cedex, France}
\email{guillot@math.unistra.fr}
\urladdr{www-irma.u-strasbg.fr/\raise-2pt\hbox{\~{}}guillot/}

\address{
(C.K.)
Universit\'{e} de Strasbourg \& CNRS,
Institut de Recherche Math\'{e}matique Avan\-c\'{e}e,
7~Rue Ren\'{e} Descartes,
67084 Strasbourg Cedex, France}
\email{kassel@math.unistra.fr}
\urladdr{www-irma.u-strasbg.fr/\raise-2pt\hbox{\~{}}kassel/}

\address{
(A.M.)
Institute of Mathematics,
University of Tsukuba,
Ibaraki 305-8571, Japan}
\email{akira@math.tsukuba.ac.jp}

\begin{document}

\begin{abstract}
Non-commutative torsors (equivalently, two-cocycles) for a Hopf algebra can be used 
to twist comodule algebras. 
We prove a theorem that affords a presentation by generators and relations 
for the algebras obtained by such twisting. 
We give a number of examples, including new constructions of the quantum affine spaces
and the quantum tori.
\end{abstract}

\maketitle

\section*{Introduction}

In this introduction we work for simplicity over a field~$k$ whereas
in the text $k$~denotes an arbitrary commutative ground ring.

The starting point of this work is the following classical
construction in algebraic geometry, and its non-commutative counterpart.
First recall that a \emph{torsor} for an algebraic group~$G$ is a right
$G$-variety~$T$ such that the map
\begin{equation}\label{eq-torsor} T\times G \longrightarrow T\times
T\, ; \quad (t,g)\mapsto (t,tg)
\end{equation} is an isomorphism.

One of the uses of torsors is in \emph{twisting $G$-varieties}.
Assuming that $G$ acts on the left on~$X$ and that $T$ is a torsor, we
may form the quotient
\[ {}_T\! X = G \, \backslash (T \times X) \, .
\]
with respect to the action $g \cdot (t, x) = (tg^{-1}, gx)$. 
For example, when $G= PGL_n$ and $X=\mathbb{P}^n$ is a projective
space, the varieties $_T\! X$ are the famous Severi-Brauer varieties.

In case $T$ is a bitorsor, in the sense that it possesses an
additional left action of~$G$ turning~$T$ into a torsor and commuting
with the right action, then ${}_T\!X$ has again a left action, induced
by~$g\cdot(t, x) = (g t, x)$. Thus the twisting process may be iterated.

It is important to note that, whenever $k$ is algebraically closed,
there is only one $G$-torsor up to isomorphism, namely $G$ itself with
its action by right translations.  This torsor is ``trivial'' in the
sense that ${}_T\! X = X$.  It follows for a general field~$k$
that~${}_T\!X$ and~$X$ are ``forms'' of one another, that is, they
become isomorphic over the algebraic closure~$\bar k$ of~$k$.  
Torsors are thus genuinely arithmetic objects.

We also recall that the $G$-torsors over~$k$ can be described in terms
of cocycles for the Galois group~$\Gal(\bar k/k)$, and in this way the
set of their isomorphism classes can be identified with the
non-abelian Galois cohomology set~$H^1(k, G)$ (see~\cite[Chap.~3]{Se}).  
Likewise, the set of isomorphism classes of $G$-bitorsors is in bijection 
with~$H^1(k, Z(G))$, which is a group (here $Z(G)$ is the centre of $G$).

When $G$ is a finite group acting on an affine variety, the quotient
always exists and is affine (see~\cite[Prop.~6.15]{borel}): indeed, if
$A$ is the coordinate ring of~$X$, then that of $X/G$ is~$A^G$.
Therefore, it makes sense to state the above in terms of
algebras, and one may wish to extend it to non-commutative algebras.
This has been done by several authors using a certain twisting procedure; 
see \cite{AEGN, doi, Mon2, schauen}.
Various properties preserved under such twistings have been investigated in~\cite{AEGN, Mon2}.

The natural setting in the non-commutative case is that of Hopf algebras. 
Recall that for any Hopf algebra~$H$ the analogue of a torsor
for~$H$ is a so-called ``cleft Galois object''; see~\cite[\S~7.2]{Mon}.  
A Galois object is defined as an algebra satisfying~\eqref{eq-torsor}
after an obvious translation from varieties to algebras of functions;
the ``cleft'' condition on the other hand is a technical requirement
that is always satisfied when the Hopf algebra is finite-dimensional or pointed.  
Bitorsors correspond to ``lazy bicleft biGalois objects'' for~$H$.  
The latter have a cohomological counterpart: 
just as bitorsors correspond to elements of~$H^1(k,Z(G))$, 
the lazy bicleft biGalois objects are by~\cite{BC} in one-to-one correspondence with the elements 
of the second lazy cohomology group~$\H_\ell^2(H)$ of~$H$, 
whose definition involves cocycles in some generalised sense, or
equivalently, invariant Drinfeld twists (we will recall the definition in Section~\ref{sec-preliminaries}).

In the case where $H$ is the Hopf algebra~$\OO_k(G)$ of $k$-valued
functions on a finite group~$G$ (which is the example we are chiefly
interested in), the relationship between~$H^1(k, Z(G))$
and~$\H^2_\ell( \OO_k(G) )$ is most clearly illustrated by the
following exact sequence, which the first and second-named authors
have obtained in~\cite{GK}:
\[ 
0 \longrightarrow H^1(k, Z(G) ) \longrightarrow \H^2_\ell(\OO_k(G)) 
\longrightarrow \H^2_\ell( \OO_{\bar k}(G) ) \longrightarrow 1 \, .
\] 
(This holds when $k$ has characteristic~$0$ and the irreducible
$\bar k$-representations of~$G$ can be realized over~$k$.)  
Thus we see that $\H^2_\ell( \OO_k(G) )$ carries a mixture of arithmetic
information coming from~$H^1(k, Z(G))$ and purely non-commutative
information coming from~$\H^2_\ell( \OO_{\bar k}(G) )$ 
(the latter is non-trivial in general, and can even be non-abelian, 
as was established in~\cite{GK}).

Our goal in the present paper is to have a new look at the twisting procedure
in the non-commutative setting, and to perform some explicit
computations; given an algebra and a torsor, we wish to describe as
precisely and as simply as possible the new algebra obtained. This is
achieved with the Presentation Theorem \ref{presentation}.

We start in Section~\ref{sec-preliminaries} by giving all necessary definitions
and by recalling the twisting construction in the context of Hopf algebras.
We are somewhat more explicit than is usual in the literature, and state
some results in the framework of braided categories.

At first sight, the twisting procedure does not seem to be the naive
non-commut\-ative analogue of the one described in this introduction. 
In Section~\ref{sec-group} however, when we specialize the general
results to the Hopf algebra $\OO_k(G)$, we are able to show
that the two definitions do indeed coincide. 
We thus provide a connection, often neglected in the ``non-commutative literature'',
between the non-commutative twisting and the above-mentioned twisting in algebraic geometry.

Section~\ref{sec-presentation} contains the Presentation Theorem. 
The latter roughly says that the ``obvious'' 
generators and relations which naturally appear for a twisted algebra
are actually sufficient, and provide a complete presentation. 
Our presentation theorem is not surprising in its general form (Theorem~\ref{presentation}),
but its explicit forms 
(Theorems~\ref{thm-presentation-group-algebra-case} and~\ref{thm-presentation})
turn out to be powerful when applied to special situations; 
in particular, they yield the familiar quantum affine spaces and quantum tori.
We feel that our method of computing  twisted algebras is quite natural and useful, 
as is demonstrated by our concrete examples. 

In Section~\ref{sec-braided}, we address the following question: 
if the algebra to be twisted is itself a Hopf algebra, do we get another Hopf algebra? 
The answer is essentially yes, although one
must consider the more general class of Hopf algebras in braided
categories to get a precise statement. These algebras have been
called ``braided groups" by Majid~(\cite{Maj, Maj2}). We give an
example of a braided group which is a twist of the algebra~$SL(2)$.

Our typical setting is with~$H= \OO_k(G)$, so that we start with
a~$G$-algebra~$A$ and obtain a new~$G$-algebra which is
denoted~$A_F$. For most of the article, in a way, we are concerned
with properties of~$A_F$ which are~$G$-equivariant in some sense or
other. Typically these can be formally derived from corresponding
properties of~$A$. However one may forget the~$G$-action and
study~$A_F$ as an algebra in its own right. This we do in
Section~\ref{sec-u-sl2}: there we study a twist of the universal
enveloping algebra of~$\slt$ and show that its representation theory
is considerably more involved than the well-known representation
theory of~$\slt$.  Restricting attention to those modules which also
have an action of~$G$ would not have produced anything truly new.

Finally, in the appendix we show that the twisting technique provides
a simple alternative construction of versal extensions for bicleft
biGalois objects of~$\OO_k(G)$.  Versal extensions had been
constructed by Eli Aljadeff and the second-named author in~\cite{AK}
for cleft Galois objects of a large family of Hopf algebras, including
all finite-dimensional ones.

\section{Preliminaries}\label{sec-preliminaries}

In this section we recall some Hopf-theoretic constructions. We shall
give the definition of a two-cocycle~$\sigma $ for a Hopf algebra~$H$
and explain how to build certain algebras~${}_{\sigma } H$. Further,
we explain how an~$H$-comodule algebra~$A$ can be twisted into an
object~${}_{\sigma } A$, which is again an~$H$-comodule algebra
when~$\sigma $ is ``lazy''. What is more, these two operations agree
when~$A = H$. We prove the first general properties, and introduce the
vocabulary of Drinfeld twists.

The material in this section builds on work of Doi \cite{doi} and
Schauenburg \cite{schauen}. In particular the Hopf algebra~${}_\sigma
H_{\sigma ^{-1}}$ presented below was investigated by Doi and
rediscovered by Schauenburg, who also established the monoidal
equivalence of Section~\ref{subsec-twist-Hopf}.

We fix a non-zero commutative ring~$k$ over which all our
constructions are defined.  In particular, all linear maps are
supposed to be $k$-linear and unadorned tensor products mean tensor
products over~$k$.

\subsection{Hopf algebras}

We assume that the reader is familiar with the language of Hopf
algebras. In the sequel, the coproduct of a Hopf algebra will always
be denoted by~$\Delta$, and its counit by~$\eps$.  We use the
Heyneman-Sweedler sigma notation (see \cite[Sect.~1.2]{Sw}):
\[
\Delta(x) = \sum_{(x)} x_1 \otimes x_2
\]
for the coproduct of an element~$x$ of the Hopf algebra.

Assume that $H$ is a Hopf algebra. We let ${}^{H}\mathcal{M}$ denote
the $k$-linear additive category of left $H$-comodules.  This category
is abelian if $H$ is a flat module over~$k$.  
The category~${}^{H}\mathcal{M}$ naturally forms a $k$-linear monoidal category.
We denote the $H$-coaction on every $M \in {}^{H}\mathcal{M}$ by 
\[
M \to H \otimes M \, ; \;\;
m \mapsto \sum_{(m)} \, m_{-1} \otimes m_0  \, .
\]  
Recall that the tensor product of $M$, $N$ in ${}^{H}\mathcal{M}$ is the
$k$-module $M \otimes N$ equipped with the $H$-coaction 
\[
m \otimes n \mapsto \sum_{(m),(n)} \, m_{-1}n_{-1} \otimes m_0 \otimes n_0 \, .
\]

\subsection{Cocycles and Galois objects}\label{subsec-two-cocycles-on-Hopf-algebras}

A \emph{two-cocycle} of a Hopf algebra~$H$ is a bilinear form~$\sigma
: H \times H \to k$ satisfying the equations
\begin{equation}\label{2cocycle} 
\sum_{(x),(y)} \,\sigma(x_{1}, y_{1}) \, \sigma(x_{2}y_{2} , z) 
= \sum_{(y),(z)} \, \sigma(y_{1} , z_{1}) \, \sigma(x , y_{2} z_{2})
\end{equation} for all $x,y,z \in H$. The only two-cocycles that we
shall consider in this paper will be convolution-invertible, so that
we may as well consider invertibility as part of the definition.

Two two-cocycles $\sigma $ and $\tau$ are called \emph{equivalent} if
there exists a convolution-invertible linear form $\lambda : H \to k$
such that for all $x,y \in H$ we have
\begin{equation}\label{cohomologous} \tau (x, y) = \sum_{(x),(y)} \,
\lambda(x_1) \, \lambda(y_1) \, \sigma(x_2, y_2) \,
\lambda^{-1}(x_3y_3) \, ,
\end{equation} where $\lambda^{-1}$ is the convolution inverse
of~$\lambda$.

Assume that $\sigma : H \times H \to k$ is a two-cocycle of~$H$.  We
define a right $H$-comodule algebra ${}_{\sigma} H$ as follows: its
underlying right $H$-comodule is~$H$ itself and as an algebra its
product is given by
\begin{equation*}
x *_{\sigma} y = \sum_{(x),(y)} \, \sigma(x_1 , y_1) \, x_2 \, y_2
\end{equation*} 
for all $x, y \in H$. The unit of ${}_{\sigma}H$ is
given by $\sigma(1, 1)^{-1}$.  One checks that equivalent two-cocycles
give rise to isomorphic comodule algebras (see \eqref{alg-isom}
below), and vice versa.

Right $H$-comodule algebras of the form ${}_{\sigma} H$ can be
characterized as the \emph{cleft (right) $H$-Galois objects}, the
non-commutative analogues of $G$-torsors alluded to in the
introduction; see~\cite[\S~7.2]{Mon} for this.  For all purposes
of this paper however, we may as well define a cleft $H$-Galois
object to be a right $H$-comodule algebra of the form~${}_{\sigma} H$.
The mirror-image construction gives those left $H$-comodule algebras
which can be characterized as \emph{cleft left $H$-Galois objects};
for this construction we should replace two-cocycles with the opposite
sided version, i.e., with bilinear forms $\nu$ on $H$ satisfying
\begin{equation*} 
\sum_{(x),(y)} \, \nu(x_{1}y_{1} , z) \, \nu(x_{2},
y_{2})= \sum_{(y),(z)} \, \nu(x , y_{1} z_{1}) \, \nu(y_{2} , z_{2}) \, ;
\quad \mbox{cf. \eqref{2cocycle}}.
\end{equation*}

\subsection{Twisting Hopf algebras and comodule algebras}\label{subsec-twist-Hopf} 

Assume that the bilinear form $\sigma : H\times H \to k$ 
is a two-cocycle of a Hopf algebra $H$.  Let $\sigma^{-1}$ denote
the convolution inverse of $\sigma$.  We define a bialgebra
${}_{\sigma}H_{\sigma^{-1}}$ as follows:\ its underlying coalgebra is
$H$ itself and as an algebra its product is given by
\begin{equation}\label{twistedproductofHopf} 
x \star_{\sigma} y =
\sum_{(x),(y)} \, \sigma(x_1 , y_1) \, x_2 \, y_2 \, \sigma^{-1}(x_3 , y_3)
\end{equation}
for all $x, y \in H$. This ${}_{\sigma}H_{\sigma^{-1}}$
indeed forms a Hopf algebra with respect to the original counit of~$H$
and the twisted antipode as given in~\cite[Th.~1.6\,(b)]{doi}.

Let us write $H' = {}_{\sigma}H_{\sigma^{-1}}$.  Since $H = H'$ as
coalgebras, every left $H$-comodule $M \in {}^{H}\mathcal{M}$ can be
regarded as a left $H'$-comodule in the obvious manner. Let
${}_{\sigma}M \in {}^{H'}\mathcal{M}$ denote the resulting object. It
is known that $M \mapsto {}_{\sigma}M$ gives a $k$-linear isomorphism
${}^{H}\mathcal{M} \overset{\cong}{\longrightarrow} {}^{H'}\mathcal{M}$ 
of monoidal categories 
when combined with the following isomorphisms in~${}^{H'}\mathcal{M}$:
\begin{multline}\label{monoidstruc1} 
\xi = \xi_{M,N} : {}_{\sigma}M \otimes {}_{\sigma}N
 \overset{\cong}{\longrightarrow} {}_{\sigma}(M \otimes N) \, ;\\ 
 \xi(m \otimes n) = \sum_{(m),(n)} \sigma(m_{-1}, n_{-1}) \, m_0 \otimes n_0 \, ,
\end{multline}
where $M,\ N \in {}^{H}\mathcal{M}$, and the isomorphism
\begin{equation}\label{monoidstruc2} 
\eta : k \overset{\cong}{\longrightarrow} {}_{\sigma}k\, ;
\quad \eta(1) = \sigma(1, 1)^{-1} .
\end{equation}
We remark that the convolution inverse $\sigma^{-1}$ of $\sigma$, 
regarded as a bilinear form $H' \times H' \to k$, is a two-cocycle of $H'$, such that
${}_{\sigma^{-1}}H'_{\sigma} = H$. We have the obvious functor
${}^{H'}\mathcal{M} \to {}^{H}\mathcal{M}$ denoted by $P \mapsto
{}_{\sigma^{-1}}P$, just as above. This together with the monoidal
structures defined in the same way as~\eqref{monoidstruc1},
\eqref{monoidstruc2}, with $\sigma$ replaced by~$\sigma^{-1}$, 
gives an inverse of the monoidal isomorphism ${}^{H}\mathcal{M} \overset{\cong}{\longrightarrow}
{}^{H'}\mathcal{M}$ above.

The monoidal isomorphism preserves algebra objects, i.e., comodule
algebras.  More explicitly, if $A$ is a left $H$-comodule algebra,
then ${}_{\sigma}A$ is a left $H'$-comodule algebra with respect to
the product defined by
\begin{equation}\label{twistedproduct-of-A} 
a *_{\sigma} b = \sum_{(a),(b)} \, \sigma(a_{-1} , b_{-1}) \, a_0 \, b_0
\end{equation} 
for all $a, b \in A$. 

We may set $A = H$ in the situation above; this is possible since $H$ is a
left $H$-comodule algebra. Now, notice that we have a monoidal
isomorphism ${}^{H}\mathcal{M}^H \cong {}^{H'}\mathcal{M}^H$ between
the bicomodule categories, in the same way as above.  Since $H$ is in
fact an $(H, H)$-bicomodule algebra, ${}_{\sigma}H$ is an $(H',
H)$-bicomodule algebra; it is a cleft Galois object on both
sides. This ${}_{\sigma}H$, regarded as just a right $H$-comodule
algebra, coincides with what was constructed in Section~\ref{subsec-two-cocycles-on-Hopf-algebras}.

Assume that $\tau$ is another two-cocycle of $H$, and set $H'' =
{}_{\tau}H_{\tau^{-1}}$.  Assume that $\sigma$ and $\tau$ are
equivalent, and choose a convolution-invertible linear form $\lambda$
on $H$ satisfying \eqref{cohomologous}. Let $A$ be a left $H$-comodule
algebra.  As is directly verified,
\begin{equation*}
f : H'' \to H', \quad f(x) = \sum_{(x)} \, \lambda(x_1) \, x_2 \, \lambda^{-1}(x_3)
\end{equation*}
is a Hopf algebra isomorphism,
\begin{equation}\label{alg-isom} 
g : {}_{\tau}H \to {}_{\sigma}H, \quad g(x) = \sum_{(x)} \, \lambda(x_1) \, x_2
\end{equation} 
is an algebra isomorphism, and the diagram
\begin{equation*}
\begin{CD} {}_{\tau}A @>{}>> H'' \otimes {}_{\tau}A \\ 
@V{g} VV @VV{f\otimes g}V \\ 
{}_{\sigma}A @>{}>> H' \otimes {}_{\sigma}A,
\end{CD}
\end{equation*} commutes, where the horizontal arrows denote the
natural coactions. In particular, if $A = H$, then $g$ is an
isomorphism of right $H$-comodule algebras as well.

\begin{Rem}\label{rmk-alternative-schauenburg} 
Since ${}_{\sigma}H$ is in particular an $(H', H)$-biGalois object, 
it follows by Schauenburg~\cite[Th.~5.5]{schauen} that the cotensor product 
$ M \mapsto  {}_\sigma H \, \square_H M $ gives a $k$-linear equivalence
${}^H \mathcal{M} \overset{\cong}{\longrightarrow} {}^{H'} \mathcal{M}$
of monoidal categories; 
we observe that in our situation, we do not need the assumption posed in~\cite{schauen} 
that $H$ is flat over~$k$.
 
The structure map $M \to H \otimes M$ on $M$ induces an isomorphism 
\[
{}_{\sigma}M \overset{\cong}{\longrightarrow} {}_\sigma H \, \square_H M
\]
in ${}^{H'} \mathcal{M}$.
This shows that the isomorphism 
${}^H \mathcal{M} \overset{\cong}{\longrightarrow} {}^{H'} \mathcal{M}\, ;
\; M \mapsto {}_{\sigma}M$ 
given above
is naturally isomorphic to Schauenburg's equivalence.
It follows that if $A$ is a left $H$-comodule algebra, then the left $H'$-comodule
algebra ${}_{\sigma}A$ is isomorphic to ${}_\sigma H \, \square_H A$.

What this means is that the new multiplication on~${}_\sigma A$ can be
described alternatively as induced from the multiplication on the
tensor product of algebras~${}_\sigma H \otimes A$. In Remark
\ref{rmk-naive-def} below, we shall see that this takes a particularly
simple form in the case of~$H = \OO_k(G)$, and indeed it will become
clear that the construction of~${}_\sigma A$ is an analogue of the
classical case considered in the introduction. 
For all other purposes though, both theoretical and practical, 
we stick to the definition of~${}_\sigma A$ given by~\eqref{twistedproduct-of-A}.
\end{Rem}

\subsection{Braided structures}\label{subsec-braided-structures}

Now assume that the Hopf algebra~$H$ is cobraided (or coquasitriangular)
with universal $R$-form $r : H \times H \to k$.
This equips the monoidal category~${}^H \mathcal{M}$ of comodules with a braided structure,
the braiding
\[ 
\gamma_{M,N} : M\otimes N \to N\otimes M
\]
between two comodules $M,N$ being given by
\[
\gamma_{M,N}(m \otimes n) = \sum_{(m), (n)} \, r(n_{-1}, m_{-1}) \, n_0 \otimes m_0
\] 
for all $m \in M$ and $n\in N$ (for details, see \cite[VIII.5]{Ks}).
Note that a commutative Hopf algebra is always cobraided with
$r(x,y) = \eps(x) \eps(y)$ for all $x,y \in H$.

Then a two-cocycle~$\sigma $ on~$H$ defines a cobraided structure
on the Hopf algebra~$H' = {}_\sigma H_{\sigma^{-1}}$ 
with universal $R$-form~$r_\sigma $ given by
\[ 
r_\sigma = \sigma_{21} \, r \, \sigma^{ -1} \, ,   
\]
where~$\sigma_{21}(x, y) = \sigma(y, x)$. 
It follows that~${}^{H'} \mathcal{M}$ also possesses a braided structure.
The following is a standard result.

\begin{prop}\label{prop-equiv-braided-categories} 
The functor 
${}^H \mathcal{M} \overset{\cong}{\longrightarrow} {}^{H'} \mathcal{M}\, ;
\; M \mapsto {}_{\sigma}M$ 
defined above is an equivalence of braided monoidal categories.
\end{prop}

Following~\cite{Maj} (see also~\cite{Gur, Maj2}), we call an
algebra~$A$ in~${}^H \mathcal{M}$ \emph{braided-commutative} if
\begin{equation*}
\mu_A = \mu_A \circ \gamma_{A,A} \, ,
\end{equation*}
where $\mu_A : A \otimes A \to A$ is the product of~$A$.

\begin{coro}\label{coro-comm-gives-braided-comm}
If~$A$ is a braided-commutative algebra in~${}^H \mathcal{M}$,
then~${}_\sigma A$ is a braided-commutative algebra in~${}^{H'}
\mathcal{M}$.
\end{coro}

\subsection{Lazy cocycles and biGalois objects}\label{subsec-lazy-two-cocycles} 

Following~\cite{BC}, we say that a two-cocycle $\sigma : H \times H \to k$ 
of a Hopf algebra~$H$ is \emph{lazy} if
\begin{equation}\label{lazy2} 
\sum_{(x), (y)} \, \sigma(x_1 , y_1)\, x_2y_2 = \sum_{(x), (y)} \, \sigma(x_2 , y_2) \, x_1y_1
\end{equation} 
for all $x,y \in H$.  The set $Z^2_{\ell}(H)$ of
convolution-invertible lazy two-cocycles of~$H$ is a group for the
convolution product of bilinear forms.

We say that a convolution-invertible linear form $\lambda : H \to k$ is \emph{lazy} if
\begin{equation*}
\sum_{(x)} \, \lambda(x_1) \, x_2 = \sum_{(x)} \, \lambda(x_2) \, x_1
\end{equation*} 
for all $x\in H$. For such a $\lambda$, we define a bilinear form 
$\partial(\lambda) : H \times H \to k$ for all $x,y \in H$ by
\begin{equation*}
\partial(\lambda)(x, y) 
= \sum_{(x),(y)} \, \lambda(x_1) \, \lambda(y_1) \, \lambda^{-1}(x_2 y_2) \, ,
\end{equation*} 
where again $\lambda^{-1}$ is the convolution inverse of~$\lambda$.  
The bilinear form~$\partial(\lambda)$ is a convolution-invertible lazy two-cocycle of~$H$.  
The set~$B^2_{\ell}(H)$ of such bilinear forms is a central subgroup of~$Z^2_{\ell}(H)$.

We can therefore define the quotient group
\begin{equation*}
\H_{\ell}^2(H) = Z^2_{\ell}(H)/B_{\ell}^2(H) \, ,
\end{equation*} 
which is called the \emph{second lazy cohomology group} of~$H$.

When the two-cocycle~$\sigma$ is lazy, we deduce from~\eqref{lazy2}
that the product $x \star_{\sigma} y$ in the twisted Hopf algebra
${}_{\sigma}H_{\sigma^{-1}}$ as given in~\eqref{twistedproductofHopf}
coincides with the original product, 
so that $H = {}_{\sigma}H_{\sigma^{-1}}$ as Hopf algebras.  It follows that, 
under the laziness assumption, ${}_{\sigma} H$ is an $(H, H)$-bicomodule algebra.

As was shown in~\cite[Th.~3.8]{BC}, the group $\H_{\ell}^2(H)$
classifies the \emph{bicleft biGalois objects of~$H$} up to isomorphism,
via $\sigma \mapsto {}_{\sigma} H$. Again for our purposes we may as
well define a bicleft biGalois object as an $(H, H)$-bicomodule
algebra of the form~${}_{\sigma} H$, where $\sigma$ is a lazy two-cocycle.
 
Let $\sigma$, $\tau$ be lazy
two-cocycles of $H$, and let $\sigma \, \tau$ denote their convolution
product.  Assume that $A$ is a left $H$-comodule algebra. Since
${}_{\sigma}A$ is a left $H$-comodule algebra, we can twist it by~$\tau$ 
to obtain a left $H$-comodule algebra
${}_{\tau}({}_{\sigma}A)$. 
We see from the definitions that
\begin{equation}\label{twist-twice} 
{}_{\tau}({}_{\sigma}A) = {}_{\tau \sigma}A
\end{equation} 
as left $H$-comodule algebras. 

\subsection{The point of view of Drinfeld twists}\label{sect-twists}

Suppose now that the Hopf algebra~$H$ is \emph{finitely generated
projective} as a $k$-module, and let~$K = H^* = \Hom_k(H,k)$ be the dual
Hopf algebra with evaluation map $\langle - , -\rangle : K \otimes H \to k$. 
In these circumstances, there is another point of view on two-cocycles,
which is often simpler.

The two-cocycles of~$H$ can be expressed in terms of certain
two-tensors on~$K$.  The correspondence is as follows.  Any bilinear
form $\sigma : H \times H \to k$ on~$H$ defines a two-tensor $F\in K \otimes K$ by
\begin{equation}\label{cocycle-twist} 
\langle F , x \otimes y \rangle = \sigma(x,y)
\end{equation} 
for $x,y \in H$.  It is easy to check that $\sigma$ is
convolution-invertible (resp.\ lazy) if and only if $F$ is invertible
(resp.\ invariant) in~$K \otimes K$. 
Recall that an element $F\in H \otimes H$ is \emph{invariant} if and only
\begin{equation*} 
\Delta(a) \, F = F \, \Delta(a)
\end{equation*} 
for all $a\in K$.

Moreover, $\sigma$ is a two-cocycle of~$H$ if and only if $F$ is a
\emph{Drinfeld twist} (or simply a \emph{twist}), by which we mean
that it satisfies the following equation in $K\otimes K \otimes K$:
\begin{equation*}
(F \otimes 1) \, (\Delta \otimes \id_H)(F) = (1 \otimes F) \, (\id_H \otimes \Delta)(F) \, .
\end{equation*} 
Two twists $F$ and $F'$ are called \emph{gauge
equivalent} when the corresponding two-cocycles are equivalent, and
this happens precisely when there exists an invertible element $a \in
K$ such that
\[ 
F' = (a\otimes a) \, F \, \Delta(a^{-1}) \, .
\]

A convolution-invertible lazy two-cocycle $\sigma$ of the
form~$\partial(\lambda)$ as above corresponds to an invariant
invertible Drinfeld twist~$F$ of the form
\begin{equation}\label{triv-twist} 
F = (a \otimes a) \, \Delta(a^{-1}) \in K \otimes K \, ,
\end{equation} 
where $a$ is some invertible \emph{central} element
of~$H$.  We say that a twist of the form~\eqref{triv-twist}, with
$a$~central, is \emph{trivial}.  For details, see~\cite[Sect.~1]{GK}.

Therefore, the second lazy cohomology group~$\H_{\ell}^2(H)$ is
isomorphic to the quotient of the group of invariant invertible
Drinfeld twists on~$K$ by the central subgroup of trivial twists.

The $k$-linear abelian monoidal category~${}^H\mathcal{M}$ of
left $H$-comodules is canonically identified with the category of the
same kind consisting of right $K$-modules, which we denote by~$\mathcal{M}_K$. 
Let $A$ be a left $H$-comodule algebra, or equivalently, a right $K$-module algebra. 
Assume that a two-cocycle $\sigma$ of~$H$ and a twist $F$ on~$K$ are in correspondence 
as in~\eqref{cocycle-twist}.
In this case we write $A_F$ (resp.\ $*_F$) for ${}_{\sigma}A$ (resp.\ for~$*_{\sigma}$).
By using $F$, the twisted product $*_F$ is expressed by
\begin{equation*}
a *_F b = \mu_A((a \otimes b)F)
\end{equation*} 
for all $a, b \in A$, where $\mu_A : A \otimes A \to A$ is the product of $A$.

We shall try to keep in mind both the ``two-cocycle'' and the
``Drinfeld twist'' points of view throughout the article.

\subsection{Semisimplicity of $A_F$}\label{subsec-simplicity} 

We end this section by discussing semisimplicity, a property that under certain hypotheses is preserved under twisting.

Assume that $k$ is a field and that $K$ is a finite-dimensional cosemisimple unimodular Hopf algebra. 
For example, if $G$ is a finite group, then $K =\OO_k(G)$ with $\mathrm{char} \, k \nmid |G|$, and
$K = k[G]$ satisfy the assumption.

\begin{prop}\label{prop-F-preserves-semisimple} 
Let $F \in K \otimes K$ be a twist on $K$
and let $A$ be a finite-dimensional right $K$-module algebra.  
Then the Jacobson radical $\mathrm{Rad} \, A_F$ of~$A_F$, regarded as a subspace of~$A$, 
coincides with the Jacobson radical $\mathrm{Rad} \, A$ of~$A$. 
\end{prop}

\pf 
Set $H = K^*$, and let $\sigma$ denote the two-cocycle of~$H$
corresponding to $F$, so that $A_F = {}_{\sigma}A$. 
By the unimodularity assumption, \cite[Th.~3.13]{AEGN} shows that
$({}_{\sigma}H_{\sigma^{-1}})^*$ is cosemisimple.  
Since $K$ is finite-dimensional cosemisimple and $A$ is left Artinian, 
the opposite-sided version of \cite[Th.~1.3]{Sk} implies that $\mathrm{Rad} \, A$ is $K$-stable. 
By Proposition~3.1\,(2) of~\cite{Mon2}, $\mathrm{Rad} \, A$ is a nilpotent ideal in ${}_{\sigma}A$, 
which implies  the inclusion $\mathrm{Rad} \, A \subset \mathrm{Rad} \, {}_{\sigma}A$. 
This result on $H$, $A$, $\sigma$ applied to ${}_{\sigma}H_{\sigma^{-1}}$, ${}_{\sigma}A$,
$\sigma^{-1}$ shows the converse inclusion $\mathrm{Rad} \, {}_{\sigma}A \subset \mathrm{Rad} \, A$.  
\epf

\begin{coro}\label{coro-ss}
Under the above hypotheses,  the twisted algebra~$A_F$ is semisimple if and only if $A$~is.
\end{coro}

\section{Results for the Hopf algebras $k[G]$ \& $\mathcal{O}_k(G)$}\label{sec-group}

In this section we specialize the constructions of Section~\ref{sec-preliminaries}
to our favorite examples, the Hopf algebras~$k[G]$ and~$\mathcal{O}_k(G)$. 
In each case we recall the definitions and give some information on the group of lazy two-cocycles. 
For~$\mathcal{O}_k(G)$ we also prove that the twisted
algebras~${}_\sigma A$ of Section~\ref{sec-preliminaries} have an
alternative description which is plainly that of the beginning of the introduction with
all commutativity requirements removed.

\subsection{The Hopf algebra $k[G]$}\label{subsec-group-algebra} 

Let $G$ be a group, and let $k[G]$ denote the group algebra of $G$. We
regard $k[G]$ as a cocommutative Hopf algebra with coproduct and counit
defined for all $g\in G$ by 
\[
\Delta(g) = g \otimes g  \quad\text{and}\quad 
\varepsilon(g) = 1  \, .
\]

Regard the group $k^{\times}$ of all units in $k$ as a trivial
$G$-module. Every group two-cocycle $G \times G \to k^{\times}$ is
linearized to a two-cocycle $k[G] \times k[G] \to k$ of $k[G]$, so
that we can identify the second group cohomology~$\H^2(G, k^{\times})$
with the second lazy cohomology group~$\H_{\ell}^2(k[G])$.

Assume that the group $G$ is abelian. Let $\mathrm{Alt}^2(G,
k^{\times})$ denote the group of all \emph{alternating bicharacters}
of $G$, i.e., of all bimultiplicative maps $b : G \times G \to
k^{\times}$ such that $b(g, g) = 1$ for all $g \in G$. Every group
two-cycle $\sigma : G \times G \to k^{\times}$ defines an alternating
bicharacter $b=b_{\sigma}$ by
\begin{equation}\label{b_sigma} 
b(g, h) = \frac{\sigma(h, g)}{\sigma(g, h)} \, ,
\end{equation} 
where $g, h \in G$. 
The assignment $\sigma \mapsto b_{\sigma}$ induces a (split) group epimorphism
\begin{equation}\label{group-epi} 
\H^2(G, k^{\times}) \to \mathrm{Alt}^2(G, k^{\times}) \, .
\end{equation} 
Assume that $G$ is finitely generated, and let $N$ denote the exponent 
of the torsion subgroup of~$G$. 
It is known that if $k^{\times} = (k^{\times})^N$, 
then the epimorphism~\eqref{group-epi} is an isomorphism.

The~$k[G]$-comodules can be identified with~$G$-graded modules, which
form the category~$\MG$. We keep assuming that~$G$ is abelian, so that
this category has a trivial braiding; as explained in Section~\ref{subsec-braided-structures}, 
a two-cocycle~$\sigma $ gives a new
braided structure on the same category, and we shall write~$\MGb$ when
we want to make a distinction between the braided categories. The
notation refers to the fact that the braiding
\[ 
\gamma_{M,N} : M\otimes N \to N\otimes M
\]
is defined by $\gamma_{M,N}( m \otimes n ) = b_\sigma( g, h ) \, n
\otimes m$ whenever $m$ (resp.~$n$) is homogeneous of degree~$h$
(resp.\ of degree~$g$).

By Corollary~\ref{coro-comm-gives-braided-comm}, if~$A$ is a
commutative~$G$-graded algebra, then the twisted algebra~${}_\sigma A$
is braided-commutative in~$\MGb$. 
This means that if~$a_1$ (resp.\ $a_2$) is homogeneous of degree~$g_1$ 
(resp.\ of degree~$g_2$), then
\[
a_2 *_\sigma a_1 = b_\sigma (g_1, g_2) \, a_1 *_\sigma a_2
\, .  
\]

This corollary could also have been obtained from the next lemma,
which will be of independent interest later.

\begin{lem} \label{lem-homogeneous} Let~$A$ be a~$G$-graded algebra,
and let~$a_1$ (resp.\ $a_2$) be an element of~$A$ which is homogeneous
of degree~$g_1$ (resp.\ of degree~$g_2$). One has
\[
 a_1 *_\sigma a_2 = \sigma (g_1, g_2) \, a_1 a_2 \, .   
 \]
 \end{lem}

\begin{proof} 
This is immediate from~\eqref{twistedproduct-of-A}.
\end{proof}

\subsection{The Hopf algebra~$\mathcal{O}_k(G)$}\label{subsec-twists-for-abelian-groups}

Let $G$ be a \emph{finite} group, and let
$\mathcal{O}_k(G)$ denote the algebra of all functions on $G$.  This
$\mathcal{O}_k(G)$ is the dual Hopf algebra of $k[G]$, so that
left~$\mathcal{O}_k(G)$-comodule algebras can be identified with
right~$k[G]$-module algebras. It also follows from duality
that~$\H_{\ell}^2(\mathcal{O}_k(G))$ is identified with the group of
invariant twists on $k[G]$ divided by the central subgroup of trivial
twists.

The category of~$G$-modules always has the trivial braiding. If we
follow the procedure of Section~\ref{subsec-braided-structures}, we find
that another braiding for the same category is given by
the~$R$-matrix~$F_{21} F^{-1}$, where~$F$ is the invariant twist
corresponding to the lazy cocycle~$\sigma $.

\begin{Rem}\label{rmk-naive-def} 
Remark~\ref{rmk-alternative-schauenburg} specializes to the following. 
Let us write~$\OO$ as a shorthand for~$\mathcal{O}_k(G)$. 
Let~$\sigma $ be a lazy two-cocycle. Then~${}_\sigma \OO$
is an~$(\OO, \OO)$-bicomodule algebra, and so also a~$(G, G)$-bimodule. 
Explicitly if~$e_g$ is the Dirac function at~$g\in G$, 
the actions are given by $h \cdot e_g = e_{g h^{-1}}$ and $e_g \cdot h = e_{h^{-1} g}$.

Now let $A$ be any left $\OO$-comodule algebra, or right $G$-algebra. 
With the left action on ${}_\sigma \OO$ above turned into a right action using the antipode, 
we regard ${}_\sigma \OO \otimes A$
as a right $G$-algebra by 
\[
(e_g \otimes a) \cdot h = e_{gh} \otimes a \cdot h \, .
\] 
Let us define the algebra of $G$-invariants by
\[ 
A' = ( {}_\sigma \OO \otimes A)^G \, .
\]
This $A'$ is again a right $G$-algebra by using the right action on ${}_\sigma \OO$,
or explicitly by $(e_g \otimes a) \cdot h = e_{h^{-1}g} \otimes a$.

Now what is stated in Remark~\ref{rmk-alternative-schauenburg} is
that~$A' \cong {}_\sigma A$ as right~$G$-algebras (which of course can
also be checked directly).  This affords, at long last, the naive
definition of~${}_\sigma A$ which was announced in the introduction.
\end{Rem}

Now assume that $G$ is abelian, and that $k$ is a field which contains
a primitive $N$-th root of~$1$, where $N$ is the exponent of~$G$. 
In this favorable situation, the discrete Fourier
transform induces a Hopf algebra isomorphism
\[ 
\OO_k(G) \cong k[\widehat{G}] \, ,
\]
where $\widehat{G} = \Hom(G,k^{\times})$ is the group of characters
of~$G$. This reduces the study to the case of group algebras already
considered. In Section~\ref{subsec-abelian-groups-are-enough} below we
explain that abelian groups are almost (not quite) the only groups to consider.

Let us give some explicit formulas before we proceed. Represent an
element of~$\H^2_{\ell}(\OO_k(G))$ by a twist~$F \in k[G] \otimes
k[G]$. For any $\chi \in \widehat{G}$, let $e_{\chi}$ be the
corresponding idempotent in~$k[G]$. Expand~$F$ in the corresponding
basis of~$k[G] \otimes k[G]$:
\begin{equation*}
F = \sum_{\chi, \psi \in \widehat{G}} \, \sigma(\chi, \psi) \, e_{\chi} \otimes e_{\psi} \, .
\end{equation*}
Then this $\sigma$ is precisely the two-cocycle of $\widehat{G}$ which
corresponds to~$F$, 
and the alternating bicharacter $b_F : \widehat{G} \times \widehat{G} \to k^{\times}$ 
which arises from~$F$ is given by
\begin{equation}\label{eq-b-from-c} 
b_F(\chi,\psi) = \frac{\sigma(\psi, \chi)}{\sigma(\chi,\psi)}
\end{equation}
for all $\chi, \psi \in \widehat{G}$.

Corollary~\ref{coro-comm-gives-braided-comm} implies that if~$A$ is a
commutative~$G$-algebra, then the twisted algebra~${}_\sigma A = A_F$
is braided-commutative in~${}^{\widehat{G}} \mathcal{M}_b$. This means
that if~$a_1$ (resp.~$a_2$) is homogeneous of degree~$\chi_1$ 
(resp.\ of degree~$\chi_2$), that is, 
if~$a_i$ is an eigenvector for the action of~$G$ with associated character~$\chi_i$, then
\[
a_2 *_\sigma a_1 = b_\sigma (\chi _1, \chi _2) \, a_1 *_\sigma a_2
\, . 
\]
Likewise, Lemma \ref{lem-homogeneous} now reads as follows.

\begin{lem} \label{lem-eigenvectors} Let~$A$ be a right~$G$-algebra,
and let~$a_1$ and~$a_2$ be eigenvectors for~$G$ with
characters~$\chi_1$ and~$\chi_2$ respectively. Then in the
algebra~$A_F$ we have
\[
a_1 *_F a_2 = \sigma (\chi_1, \chi_2) \, a_1 a_2 \, .  
\]
\end{lem}

\begin{ex}\label{ex-fundamental} The following example is fundamental
in the sense that several of our subsequent examples rely on it.  Assume
that $k$ is a field whose characteristic is different from~$2$.  
Let $V= \Z/2 \times \Z/2$, and consider $V$ as a
two-dimensional vector space over the field~$\mathbb{F}_2$.  Let
$\alpha $ denote the isomorphism of groups
\[ \alpha : \mathbb{F}_2 \stackrel{\cong}{\longrightarrow} \{ -1, 1 \}
\subset k^\times \, .
\]

We define a bicharacter on~$V$ by
\[ 
b(x, y) = \alpha ( \det (x,y) ) \, .
\] 
It is easy to see that $b$ is in fact the \emph{only} non-trivial
bicharacter of~$V$ with values in~$k^\times$.  In particular, $b$ is
invariant under all automorphisms of~$V$.

In the sequel, we shall view~$b$ as a bicharacter on~$\widehat{V}$
(which of course is also isomorphic to~$\Z/2 \times \Z/2$).  The
letter~$F$ will stand for the corresponding twist, or sometimes its
equivalence class in~$\H^2_\ell( \OO_k(V))$.

Let us write down a specific choice of two-cocycle $\sigma$ on
$\widehat{V}$ such that~\eqref{eq-b-from-c} holds for $b_F = b$; 
this will be useful later, but requires some notation.  Let the elements of~$V$ be $1$,
$e_1$, $e_2$ and $e_3 = e_1e_2$, and let us write~$\widehat{e} _i$ for
the character of~$V$ whose kernel is generated by~$e_i$.  Also, let
the group~$\Z/3$ act on~$\widehat{V}$ by permuting cyclically the
three elements~$\widehat{e}_i$.

Now $\sigma$ is uniquely determined by~\eqref{eq-b-from-c} and the requirements 
(i) $\sigma(x, 1) = \sigma(1, x) = 1$ for all $x \in \widehat{V}$, 
(ii) $\sigma(\widehat{e}_i, \widehat{e}_i) = -1$ for $i=1, 2, 3$, 
(iii) $\sigma(\widehat{e}_1, \widehat{e}_2 ) = 1$, and 
(iv) $\sigma$ is $\Z/3$-invariant.  
(So for example, $\sigma(\widehat{e}_2, \widehat{e}_1) = b(\widehat{e}_1, \widehat{e}_2) \,
\sigma(\widehat{e}_1,\widehat{e}_2) = -1$.)

A direct computation shows that $\sigma$ is indeed a two-cocycle. This
particular~$\sigma$ gives us a particular twist~$F$, namely we have
\begin{eqnarray*} 
4F & = & 1 \otimes 1 - (e_1 \otimes e_1 + e_2 \otimes e_2 + e_3 \otimes e_3) \\ 
&& + (1 \otimes e_1 + e_1 \otimes 1) + (1 \otimes e_2 + e_2 \otimes 1) 
+ (1 \otimes e_3 + e_3 \otimes 1) \\
&& + (e_1 \otimes e_2 - e_2 \otimes e_1) + ( e_2 \otimes e_3 - e_3 \otimes e_2) 
+ (e_3 \otimes e_1- e_1 \otimes e_3) \, .
\end{eqnarray*}
 
Now, we may consider this~$F$ as a twist for a larger group
containing~$V$ as a subgroup, such as the alternating group~$A_4$.  By
Property\,(iv) above, $F$ is $\Z/3$-invariant for the obvious action
on~$V$, and it follows that $F$ is~$A_4$-invariant when viewed as a
twist for~$A_4$.  In~\cite[Sect.~7.5]{GK} we have proved that~$F$
represents the only non-trivial element of~$\H^2_\ell(\OO_k( A_4 )) \cong \Z/2$ 
provided $k$ is a field of characteristic~$0$.
\end{ex}

\subsection{The special role played by abelian groups}\label{subsec-abelian-groups-are-enough} 

The reader will notice that most of our examples in the sequel involve abelian
groups.  Of course we have a great control over the twists in the
abelian case, but our selection of examples is not only dictated by this.
Indeed, there are several partial results which indicate that ``most''
twists come from abelian groups anyway.  Here is such a result, which was
implicit in~\cite{GK}.

\begin{prop}\label{prop-FJ} Assume that $k$ is an algebraically closed
field of characteristic~$0$.  If $F \in k[G]\otimes k[G]$ is an
invariant twist for the finite group~$G$, then there exists an abelian
subgroup~$V$ of~$G$ and a twist $J \in k[V] \otimes k[V]$ such
that~$F$ and~$J$ are gauge equivalent.  Moreover, $V$ can be chosen to
be normal in~$G$.
\end{prop}

\begin{proof} 
Let $\varphi = F^{-1}$. Since $F$ is invariant, $\varphi$ is also a twist.  
Now consider $R_\varphi = \varphi_{21}\varphi^{-1}$ (here $\varphi_{21} = \tau (\varphi)$, 
where $\tau : k[G]\otimes k[G] \to k[G] \otimes k[G]$ is the flip).  
By~\cite[Sect.~4]{GK}, there is a normal abelian subgroup $V\subset G$
such that $R_\varphi \in k[V]\otimes k[V]$.

If we apply the Fourier transform as in Section~\ref{subsec-twists-for-abelian-groups}, 
we see that~$R_\varphi$ gives an alternating bilinear form~$b$
on~$\widehat{V}$ (this follows from a direct computation, 
which is elucidated in~\cite[Sect.~4.3]{GK}).  
By the results in Section~\ref{subsec-twists-for-abelian-groups}, 
we must have $b = b_{J'}$ for some twist $J' \in k[V] \otimes k[V]$.  
Let us prove that $F$ is gauge equivalent to $J = (J')^{-1}$.  
Note that~$J$ is a twist, since~$V$ is abelian.

Put $f = J \varphi = (J')^{-1} \varphi$; it is the product of the twist~$J$ by the 
\emph{invariant} twist~$\varphi$, and as a result, $f$~itself must be a twist.  
Moreover, we have $R_{J'} = J'_{21} (J')^{-1} = R_\varphi$ by definition, 
so that $f_{21} = f$: the twist~$f$ is symmetric.  
Now by~\cite[Cor.~3]{EG}, $f$ is gauge equivalent to $1 \otimes 1$, 
which means that there is $a \in k[G]^\times$ such that
\[ 
f = (a\otimes a) \Delta (a^{-1}) \, .
\] 
In view of the invariance of~$F$, it follows readily that
\[ 
J = fF = (a\otimes a) \Delta(a^{-1}) F = (a\otimes a) F
\Delta(a^{-1}) \, ,
\] 
which shows that $F$ and $J$ are gauge equivalent.
\end{proof}

Proposition~\ref{prop-FJ} implies that $A_F$ is isomorphic as an algebra to~$A_J$, 
and thus all the algebras obtained by our method using an invariant twist~$F$ on a finite group~$G$ 
can also be constructed using an abelian group. 
Note that the proposition does \emph{not} state that $J$ is invariant 
(so that~$A_J$ is not naturally endowed with a~$G$-algebra structure); 
and even if it were, it need not represent the same element of~$\H^2_\ell( \OO_k(G) )$ 
as~$F$ (in this case the~$G$-actions on~$A_F$ and~$A_J$ differ by an automorphism
of~$G$ as follows from~\eqref{alg-isom}).

In~\cite{GK} we have studied in detail the question of finding an invariant twist~$J$ 
lying in $k[V]\otimes k[V]$, where $V$ is a normal abelian subgroup of~$G$, 
such that $J$ and~$F$ represent the same element of~$\H^2_\ell(\OO_k(G) )$.  
While there certainly are cases where this is not possible, the counter-examples are rather exotic.  
For instance, when $G$ has odd order and only trivial conjugacy-preserving outer
automorphisms, such an invariant twist~$J$ can always be found.

\section{The presentation theorem \& examples}\label{sec-presentation}

Let $H$ be a Hopf algebra, $\sigma : H \times H \to k$ a two-cocycle, and
$A$ a left $H$-comodule algebra. 
The aim of this section is to present the twisted algebra~${}_{\sigma}A$ by generators and relations
when $A$ is given by such a presentation.

We start with the general case before specializing to the cases $H = k[G]$ and $H = \OO_k(G)$.

\subsection{The general case}\label{subsec-presentation}

Let $H$ be a Hopf algebra. 
By saying that $P$ is an $H$-\emph{subcomodule} of a left $H$-comodule~$Q$, 
we mean that the image of~$P$ by the structure map $Q \to H \otimes Q$ on~$Q$ 
sits inside the natural image of~$H \otimes P$.
In this case there is an induced linear map $Q/P \to H \otimes Q/P$ 
turning~$Q/P$ into a quotient $H$-comodule of~$Q$.

Fix a left $H$-comodule $M$. 
Then the tensor algebra $T(M)$ on~$M$ naturally forms a left $H$-comodule algebra. 
Let $R \subset T(M)$ be an $H$-subcomodule and let $\mathfrak{a} = (R)$ 
denote the ideal of~$T(M)$ generated by~$R$. Then $\mathfrak{a}$ as well is an $H$-subcomodule,
so that the quotient algebra $T(M)/\mathfrak{a}$ is a left $H$-comodule algebra.

Let $\sigma : H \times H \to k$ be a two-cocycle of $H$. 
Then we have the twisted algebras ${}_{\sigma}T(M)$ and ${}_{\sigma}(T(M)/\mathfrak{a})$.

\begin{lem}\label{TVlem} 
Regarded as a $k$-submodule of~${}_{\sigma}T(M)$, 
$\mathfrak{a}$ is again an ideal, and is generated by~$R$. Moreover,
\[
{}_{\sigma}T(M)/\mathfrak{a} = {}_{\sigma}(T(M)/\mathfrak{a}) \, .
\]
\end{lem}

\pf 
In general, let $A$ be a left $H$-comodule algebra and $L \subset A$ an $H$-subcomodule.
We see from~\eqref{twistedproduct-of-A} 
that the right (or left or two-sided) ideal of~${}_{\sigma}A$ generated by~$L$ is
included in the ideal of $A$ generated by~$L$. 
The converse inclusion follows from the equation 
\[
a \, b = \sum_{(a),(b)} \, \sigma^{-1}(a_{-1} , b_{-1}) \, a_{0} *_{\sigma} b_{0} \, ,
\] 
so that the two ideals coincide.
This proves the first half of the lemma. The remaining equality is easy to prove.  
\epf

Recall from \eqref{monoidstruc1} the isomorphism 
$\xi : {}_{\sigma}M \otimes {}_{\sigma}N \overset{\cong}{\longrightarrow}
 {}_{\sigma}(M \otimes N)$.  
 For each $r \geq 2$, let
\[
\xi_r : ({}_{\sigma}M)^{\otimes r} \overset{\cong}{\longrightarrow} {}_{\sigma}(M^{\otimes r})
\]
denote the $r-1$ times iterated operation of $\xi$ on $M^{\otimes r}$. 
Explicitly, an $r$-tensor $x_1 \otimes x_2 \otimes \cdots \otimes x_r$
in~$({}_{\sigma}M)^{\otimes r}$ is sent by $\xi_r$ to the twisted
product $x_1 *_{\sigma} x_2 *_{\sigma} \cdots *_{\sigma} x_r$ 
in~${}_{\sigma}T(M)$. 
For $r = 0, 1$, let $\xi_0 : k \to {}_{\sigma}k,\ \xi_0(1) = \sigma(1, 1)^{-1}$ 
denote the isomorphism~$\eta$ defined by~\eqref{monoidstruc2}, 
and $\xi_1 : {}_{\sigma}M \to {}_{\sigma}M$ the identity map~$\mathrm{id}_M$. 
Let
\begin{equation}\label{eqhinfty} \xi_{\infty} = \bigoplus_{r \geq 0}
\, \xi_r : T({}_{\sigma}M) \overset{\cong}{\longrightarrow} {}_{\sigma}T(M)
\end{equation} denote the direct sum of all $\xi_r$. The explicit
description of $\xi_r$ above proves the following lemma.

\begin{lem}\label{hinfty} 
The map $\xi_{\infty}$ is an isomorphism of left ${}_{\sigma}H_{\sigma^{-1}}$-comodule algebras.
\end{lem}

From the two lemmas above we now deduce the following Presentation Theorem.

\begin{thm}[Presentation Theorem, General Version]\label{presentation}
The natural map ${}_{\sigma}M \to {}_{\sigma}(T(M)/(R))$ induces an
isomorphism
\[
T({}_{\sigma}M)/(\xi_{\infty}^{-1}(R)) \overset{\cong}{\longrightarrow} {}_{\sigma}(T(M)/(R))
\]
of left ${}_{\sigma}H_{\sigma^{-1}}$-comodule algebras.
\end{thm}

\pf 
Notice from Lemma~\ref{TVlem} that 
${}_{\sigma}T(M)/(R) = {}_{\sigma}(T(M)/(R))$.  
Compose this equality with the isomorphism
\[
T({}_{\sigma}M)/(\xi_{\infty}^{-1}(R)) \overset{\cong}{\longrightarrow} {}_{\sigma}T(M)/(R)
\]
obtained from Lemma~\ref{hinfty}. Then the theorem follows.  
\epf

To apply this theorem in practice, the following result is needed.

\begin{prop}\label{hinftyinverse} 
Assume that the two-cocycle~$\sigma$ of~$H$ is lazy.  
Regard $\xi_{\infty}$ as a linear isomorphism $T(M) \overset{\cong}{\longrightarrow} T(M)$.  
Then the inverse $\xi_{\infty}^{-1}$ of~$\xi_{\infty}$ is given by
\[
\xi_{\infty}^{-1}(x_1 \otimes x_2 \otimes ... \otimes x_r) 
= x_1 *_{\sigma^{-1}} x_2 *_{\sigma^{-1}} \cdots *_{\sigma^{-1}} x_r \, ,
\]
the twisted product in ${}_{\sigma^{-1}}T(M)$, and by $\xi_{\infty}^{-1}(1)= \sigma(1, 1)$.
\end{prop}

\pf 
Fix an integer $r > 0$, and choose arbitrarily $r$ elements $x, y, z,\ldots, v, w$ of~$H$.  
By induction on~$r$, we see from the laziness
assumption that
\begin{multline*}
\sum_{(x),(y),\ldots,(v)} \sigma(x, y)\, \sigma(xy, z)\cdots\sigma(xyz\cdots v, w) \\
= \sum_{(x),(y),\ldots,(v)} \sigma(xyz\cdots v, w)\cdots\sigma(xy, z)\, \sigma(x,y) \, .
\end{multline*}
Here we have omitted the subscripts $1, 2,...$, which should be
added to each of $x, y,\ldots, v$ in correct order, i.e., from the left to the right. 
For example, when $r = 4$, the last equation implies that
\[
\xi_4 = (\xi_{M,M} \otimes \mathrm{id}_M \otimes \mathrm{id}_M) 
\circ (\xi_{M \otimes M, M} \otimes \mathrm{id}_M)
\circ \xi_{M \otimes M \otimes M, M} \, .
\] 
This equation can be generalized to any~$r$ in the obvious manner, and the generalized
equation implies the proposition.  
\epf

\subsection{The group algebra case; the quantum affine spaces and tori} 

Assume that $H = k[G]$ is the algebra of a group~$G$. 
Let $\sigma : G \times G \to k^{\times}$ be a group two-cocycle.

Our setup is now that of an algebra
\[ 
A = k\langle x_1, \ldots, x_n \rangle /(p_i)_{i \in I}
\]
which is $G$-graded, generated by homogeneous elements $x_j$ ($1 \leq j \leq n$) 
of degree~$g_j \in G$, satisfying the relations $p_i = 0$ ($i \in I$).

We shall adopt a notational convention, here and in the rest of the
paper. The generator~$x_i$, when seen as an element of~${}_\sigma A$,
will be written~$X_i$. Words in capital letters will always refer to
the twisted multiplication in~${}_\sigma A$, so for example~$X_i X_j$
always means~$x_i *_\sigma x_j$. Note that, on the other hand, we
do not distinguish between~$x_i \in k\langle x_1, \ldots , x_n
\rangle$ and the element~$x_i \in A$ (and likewise for~$X_i$).

Given an integer $r > 1$ and an $r$-tuple $(i_1, \ldots, i_r)$ of integers $1 \leq i_s \leq n$, 
define a scalar $\kappa_{i_1, \ldots, i_r}$ by
\[ 
\kappa_{i_1, \ldots, i_r} = \sigma(g_{i_1}, g_{i_2})^{-1}\sigma(g_{i_1}g_{i_2}, g_{i_3})^{-1} \cdots \,
\sigma(g_{i_1}g_{i_2} \cdots g_{i_{r-1}}, g_{i_r})^{-1} \, .
\] 
Then we have a linear isomorphism 
$\kappa : k\langle x_1, \ldots, x_n \rangle \overset{\cong}{\longrightarrow} 
k\langle X_1, \ldots, X_n \rangle$ given by
\begin{eqnarray*} 
&\kappa(x_{i_1}x_{i_2} \ldots x_{i_r}) =
\kappa_{i_1, i_2, \ldots, i_r} \, X_{i_1}X_{i_2} \ldots X_{i_r} \, , \\ 
&\kappa(x_j) = X_j \, , \quad \kappa(1) = \sigma(1, 1) \, .
\end{eqnarray*}

The following is a consequence of Theorem~\ref{presentation}, Proposition~\ref{hinftyinverse}, 
and Lemma~\ref{lem-eigenvectors}.

\begin{thm}[Presentation Theorem, Group Algebra Case]\label{thm-presentation-group-algebra-case} 
The left $H$-comodule algebra~$A$ is twisted by the two-cocycle~$\sigma$, so that
\[ 
{}_{\sigma}A = k\langle X_1, \ldots, X_n \rangle /(\kappa(p_i))_{i \in I} \, .
\] 
If $G$ is abelian and if the commutativity relation $x_ix_j = x_jx_i$ ($i \neq j$) holds in~$A$,
then in~${}_{\sigma}A$ we have
\begin{equation*}\label{twisted-commutativity-realation} 
X_iX_j = b_{\sigma}(g_j, g_i) \, X_jX_i \, ,
\end{equation*} 
where $b_{\sigma}$ is as given by~\eqref{b_sigma}.
\end{thm}

\begin{ex}\label{quantum affine space} 
As an application of Theorem~\ref{thm-presentation-group-algebra-case}, we show that the so-called 
``quantum affine spaces'' and ``quantum tori'', which are ubiquitous in quantum group theory
(see e.g.~\cite{Ta, Wa}), can be obtained as twists of the classical affine spaces and tori.

In what follows, let $n$ be an integer $\geq 1$, and assume that $G= \Z^n$ 
with canonical basis $(e_1, e_2, \ldots, e_n)$.  
We write $\d e_i: G \to \Z$ for the $i$-th~projection and consider
\[ 
\omega = - \sum_{i < j} \alpha_{ij} \, \d e_i \wedge \d e_j \, ,
\] 
where $\alpha_{ij}\in \Z$.  
Thus $\omega : G\times G \to \Z$ is an alternating bilinear form over~$\Z$.

Pick arbitrarily an element $q \in k^{\times}$.  We set
\begin{equation}\label{bicharZr} 
b(x,y)= q^{\omega(x,y)}
\end{equation} 
for $x,y\in G$, thus obtaining the alternating bicharacter $G\times G \to k^\times$ 
that takes the values $b(e_i, e_j)= q^{-\alpha_{ij}}$.  
Choose a group two-cocycle
$\sigma : G \times G \to k^{\times}$ such that~$b = b_{\sigma}$.

Regard the commutative polynomial algebra $A = k[x_1, \ldots, x_n]$ as
a $G$-graded algebra by declaring each generator~$x_i$ to be homogeneous
of degree~$e_i$. We conclude from Theorem~\ref{thm-presentation-group-algebra-case} that
\[ 
{}_{\sigma}A = k\langle X_1, \ldots, X_n \rangle /(X_iX_j - q^{\alpha_{ij}}X_jX_i)_{i < j} \, .
\] 
Thus, ${}_{\sigma}A$ turns out to be a quantum affine space.
\end{ex}

\begin{ex}\label{quantum torus} 
Regard the coordinate ring of the torus
\[ 
B = k[x_1, \ldots, x_n, y_1, \ldots, y_n]/(x_iy_i - 1)_{1 \leq i \leq n}
\] 
as a $G$-graded algebra by declaring each~$x_i$ (resp.\ each~$y_i$)
to be homogeneous of degree~$e_i$ (resp.\ of degree~$e_i^{-1}$).  
As above, we conclude that
\[ 
{}_{\sigma}B = k\langle X_1, \ldots, X_n, Y_1, \ldots, Y_n \rangle /\mathfrak{a} \, ,
\] 
where $\mathfrak{a}$ is the two-sided ideal generated by
\[ 
X_iX_j - q^{\alpha_{ij}} X_j X_i \, , \quad Y_iY_j - q^{\alpha_{ij}} Y_j Y_i \, , \quad X_iY_j - q^{-\alpha_{ij}} Y_jX_i\, , 
\qquad (i<j)
\]
$X_iY_i - Y_iX_i$ and finally $X_i Y_i - \sigma(e_i, e_i^{-1}) \, \sigma(1, 1)$.
Whatever the precise value of~$\sigma(e_i, e_i^{-1}) \, \sigma(1, 1)$ is, we have
\[ 
{}_{\sigma}B = 
k\langle X_1^{\pm 1}, \ldots, X_n^{\pm 1}\rangle /(X_iX_j - q^{\alpha_{ij}} X_j X_i)_{i < j} \, .
\] 
This is a quantum torus.
\end{ex}

\subsection{The case of~$\OO_k(G)$; the quantum tetrahedron}

We now turn to the case of the Hopf algebra~$\OO_k(G)$
(all subsequent examples of comodule algebras will be over this Hopf algebra).
Since we aim at explicit computations, it is tempting to restrict attention to abelian groups;
however, we can work in the following slightly more general setting at virtually no cost.

Our setup features an abelian normal subgroup~$V$ of~$G$. 
Next, we let~$F$ be an invariant twist for $G$ such that $F\in k[V]\otimes k[V]$, 
a condition that is easier to state than with the corresponding two-cocycle~$\sigma $. 
Finally, the alternating bicharacter for~$\widehat{V}$ as in
Section~\ref{subsec-twists-for-abelian-groups} will be denoted simply by~$b$. 
The case $G = V$ is typical, and we advise the reader to keep this particular case in mind%
---in general, the only difference is that, when we are done with twisting an algebra, 
the outcome will be a $G$-algebra and not just a $V$-algebra. 
All the other computations performed when finding a presentation only involve the subgroup~$V$.

The type of algebra which occurs in this variant of the Presentation Theorem is one of the form
\[ 
A = k\langle x_1, \ldots, x_n \rangle /(p_i)_{i \in I} \, ,
\]
where each generator~$x_i$ is an eigenvector for~$V$ with character~$\chi_i$. 
We remark that one can always find such a presentation for a $G$-algebra~$A$, 
provided that it is finitely generated as an algebra and that~$k$ is large enough a field; 
any finite set of generators is included in a finite-dimensional $k[G]$-submodule, 
which is necessarily a sum of one-dimensional submodules over the split
commutative semisimple algebra~$k[V]$.

Given an integer $r > 1$ and an $r$-tuple $(i_1, \ldots, i_r)$ of
integers $1 \leq i_s \leq n$, define a scalar~$\kappa_{i_1, \ldots, i_r}$ by
\[ 
\kappa_{i_1, \ldots, i_r} = \sigma(\chi_{i_1}, \chi_{i_2})^{-1}
\sigma(\chi_{i_1}\chi_{i_2}, \chi_{i_3})^{-1} \cdots \, 
\sigma(\chi_{i_1}\chi_{i_2} \cdots \chi_{i_{r-1}}, \chi_{i_r})^{-1} \, .
\]
Then we have a linear isomorphism 
$\kappa : k\langle x_1, \ldots, x_n \rangle \overset{\cong}{\longrightarrow} 
k\langle X_1, \ldots, X_n \rangle$ given by
\begin{eqnarray*} 
&\kappa(x_{i_1}x_{i_2} \ldots x_{i_r}) = \kappa_{i_1, i_2, \ldots, i_r} \, X_{i_1} X_{i_2} \ldots X_{i_r} \, , \\ 
&\kappa(X_j) = X_j \, , \quad \kappa(1) = \sigma(1, 1) \, .
\end{eqnarray*}

\begin{thm}[Presentation Theorem, Function Algebra Case]\label{thm-presentation} 
The twisted algebra~$A_F = {}_\sigma A$ is presented as
\[ 
A_F = k\langle X_1, \ldots, X_n \rangle /(\kappa(p_i))_{i \in I} \, .
\]
If $x_ix_j = x_jx_i$ holds in~$A$, then in~$A_F$ we have
\begin{equation*}
X_iX_j = b_{\sigma}(\chi_j, \chi_i) \, X_jX_i \, .
\end{equation*}
\end{thm}

\begin{ex} Our first example is an exercise for the reader. 
Show that, when~$q$ is specialized to a primitive~$r$-th root of unity, 
the quantum affine space as above can be obtained as a twist of the
algebra~$k[x_1, \ldots, x_n]$ equipped with a certain action of~$G = V = (\Z/r )^n$.
\end{ex}

\begin{ex}[The Quantum Tetrahedron] 
Our next example is in some sense a subspace of the quantum $3$-space
with $q= -1$, $n=3$ and $\alpha_{ij} = 1$.  
It should make the relations~$\kappa(p_i)$ appearing in the Presentation Theorem more concrete.

Let $k$ be a field of characteristic zero.  By ``tetrahedron'' we mean the
union of the four planes in~$k^3$ defined by the equations
\begin{eqnarray*} 
p_1 = x+y+z - 1 \, , && p_2 = x-y-z - 1 \, , \\ 
p_3 = -x-y+z - 1 \, , && p_4 = -x+y-z - 1 \, .
\end{eqnarray*} 
In terms of the algebra of functions on the tetrahedron, our actual object of study is the algebra
\[ A = k[x,y,z] / (p_1 p_2 p_3 p_4) \, .
\] (Thinking of the tetrahedron as embedded in~$\R^3$ with equilateral
faces and barycentre at the origin, one obtains the coordinates above
by choosing basis vectors on the axes of the three rotations of
angle~$\pi$ preserving the tetrahedron.)

The alternating group $G = A_4$ acts naturally on the tetrahedron.
The permutations $e_1= (1,2)(3,4)$, $e_2= (1,3)(2,4)$, and $\gamma = (1,2,3)$
generate~$A_4$, while~$e_1$ and~$e_2$ alone generate a copy of
Klein's \emph{Vierergruppe}~$\Z/2 \times \Z/2$. 
Our notation is going to be consistent with that of Example~\ref{ex-fundamental}. 
The actions on~$A$, or rather on the restriction to the $3$-dimensional subspace
spanned by the functions $x$, $y$, $z$, are given by the matrices
\begin{equation}\label{eq-action-a4-tetrahedron} 
e_1 =
\left( \begin{array}{ccc} 
1& 0& 0\\ 
0& -1 & 0\\ 
0& 0 & -1
\end{array}\right) , \quad 
e_2 = 
\left( \begin{array}{ccc} 
-1& 0& 0\\ 
0& -1 & 0\\ 
0& 0 & 1 \end{array}\right) , \quad 
\gamma =
\left( \begin{array}{ccc} 
0& 0& 1\\ 
1& 0 & 0\\ 
0& 1 & 0 \end{array}\right) .
\end{equation}

Twisting the algebra~$A$ 
with the help of the Drinfeld twist~$F$ of Example~\ref{ex-fundamental}, 
we obtain the twisted algebra~$A_F$, for which we are going to give a presentation.

Let $X = \kappa(x)$, $Y= \kappa (y)$ and $Z = \kappa (z)$. We shall
need the following explicit computations.

\begin{lem} 
We have the following identities in $A_F$:
\[
\begin{array}{llll} 
XY = -\kappa(xy) \, , & YX = \kappa(xy) \, , & XZ = \kappa(xz) \, , & ZX= -\kappa(xz) \, , \\ 
ZY = \kappa(yz) \, , & YZ = -\kappa(yz) \, , & X^2 = - \kappa(x^2) \, , & Y^2 = - \kappa(y^2) \, , \\ 
Z^2 = - \kappa(z^2) \, , & XYZ = \kappa(xyz) \, , & X^4 = \kappa(x^4) \, , & Y^4 = \kappa(y^4) \, , \\ 
Z^4 = \kappa(z^4) \, , & X^2Y^2 = \kappa(x^2y^2) \, , & Y^2Z^2 = \kappa(y^2z^2) \, , & 
X^2Z^2 = \kappa(x^2z^2) \, .
\end{array}
\]
\end{lem}

\begin{prop} 
The algebra $A_F$ is isomorphic to the algebra generated by three generators $X$, $Y$, $Z$ 
subject to the following relations:
\begin{equation*} 
XY = -YX \, , \quad XZ = -ZX \, , \quad YZ = - ZY \, ,
\end{equation*} 
and
\begin{multline*} 
2(X^2 + Y^2 + Z^2) - 2( X^2Y^2 + X^2Z^2 + Y^2Z^2) \\
- 8XYZ + (X^4 + Y^4 + Z^4) + 1 = 0 \, .
\end{multline*} 
The actions of~$e_1$, $e_2$ and~$\gamma $ are given by
the matrices in~\eqref{eq-action-a4-tetrahedron} with respect to the basis~$X$, $Y$, $Z$.
\end{prop}

It follows from~\eqref{twist-twice} and the isomorphism
$\H^2_\ell(\OO_k(A_4)) \cong \Z/2$ that $(A_F)_F \cong A$. Therefore,
the element of order two in~$\H^2_\ell(\OO_k(A_4))$ exchanges the
function algebra of the tetrahedron and the non-commutative algebra~$A_F$.  
Thus the tetrahedron and the ``quantum tetrahedron''~$A_F$ are mirror images of each other.
\end{ex}

\section{A braided group}\label{sec-braided}

\subsection{Twisting Hopf algebras}

Let us briefly return to the general setting of Section~\ref{sec-preliminaries}. 
We let~$H$ be a cobraided Hopf algebra and~$H' = {}_\sigma H_{\sigma ^{-1}}$. 
We know that there is an equivalence of braided monoidal categories
${}^H \mathcal{M} \cong {}^{H'} \mathcal{M}$.

There are many formal consequences of this, and we have already been
exploiting at length the fact that algebra objects are preserved by
this equivalence. Now let us look at Hopf algebras and their modules. 

First we need to recall that two algebras~$A$ and~$B$ in a braided
category may be tensored. Indeed, if $\mu_A$ (resp.~$\mu_B$) is the
product of~$A$ (resp.\ of~$B$), then the product on~$A \otimes B$ is
\[
(\mu_A \otimes \mu_B) \circ (\id_A \otimes \gamma_{B,A} \otimes \id_B) \, .
\]
(Here $\gamma_{B, A} : B\otimes A \to A \otimes B$ is the braiding.)
We shall use the notation~$A \bt B$ for emphasis
(with~$b$ for braiding in general, and also for bicharacter for the applications in this paper). 
Indeed in many situations, the underlying category has also a ``trivial'' braiding, 
so that the notation~$A \otimes B$ could be misunderstood.

Keeping in mind that the definition of a Hopf algebra~$A$ in a braided
monoidal category involves a map of algebras 
\[
\Delta \colon A  \rightarrow A \bt A \, , 
\]
it is now formal that an equivalence of braided categories has to preserve Hopf algebras. 

The notion of an~$A$ module in~${}^H \mathcal{M}$ is the obvious one. 
For example when~$H=\OO_k(G)$, so that the category in question is that
of~$G$-modules, then an~$A$-module~$M$ is, by definition,
simultaneously a~$G$-module and an~$A$-module 
and we have the compatibility condition
\[
g \cdot (a m)= (g \cdot a) (g \cdot m)
\]
for all $g\in G$, $a\in A$, and $m\in M$.

Suppose that~$A$ and~$B$ are algebras in a braided category, and
that~$M$ is an $A$-module and~$N$ a $B$-module, still in that category. 
Then $M\otimes N$ is an $A \bt B$-module in the natural way:
if $\theta_A: A\otimes M \to M$ is the action (and similarly for~$B$ and~$N$), 
then the action $A\bt B \otimes (M\otimes N) \to M\otimes N$ is given by
\[
(\theta_A \otimes \theta_B) \circ (\id_A \otimes \gamma_{B, M} \otimes \id_N)\, .
\]
In particular, if~$A$ is a Hopf algebras, then its modules may be tensored. 

In summary, we have the following.

\begin{prop}\label{prop-twist-hopf}
If ~$A$ is a Hopf algebra in~${}^H \mathcal{M}$, then~${}_\sigma A$ is
a Hopf algebra in~${}^{H'} \mathcal{M}$. 

The equivalence~${}^H \mathcal{M}\cong {}^{H'} \mathcal{M}$ induces an
equivalence between the categories of Hopf algebras in these categories. 
For each Hopf algebra~$A$, there is also a monoidal equivalence 
between the categories of~$A$-modules and the category of~${}_\sigma A$-modules.
\end{prop}

\subsection{An example related to the group $SL_2$}\label{subsec-func-on-sl2}

In order to illustrate the above phenomena, we shall now give some
explicit computations with the algebra of functions on the algebraic group~$SL_2$. 
We work with~$H= \OO_k(G)$ for~$G=\Z/2 \times \Z/2$. 

We let the group $G = \langle e_1, e_2 \rangle$ act on the algebra
\[
A= SL(2) = k[a,b,c,d]/(ad - bc -1)
\]
by 
\[
e_1 \cdot 
\left(\begin{array}{cc}
a & b \\
c & d
\end{array}\right) 
=
\left(\begin{array}{cc}
d & c \\
b & a
\end{array}\right) 
\quad\text{and}\quad 
e_2 \cdot 
\left(\begin{array}{cc}
a & b \\
c & d
\end{array}\right) 
=
\left(\begin{array}{cc}
a & -b \\
-c & d
\end{array}\right) \, . 
\]
(Each matrix identity is shorthand for four identities in~$A$.)

Let $F$ be the twist introduced in Example~\ref{ex-fundamental};
together with the algebra~$A$, it defines a twisted algebra~$A_F$,
which we also denote by~$SL_F(2)$.  Let us give a presentation of the latter.

Set $x= (a+d)/2$, $y= (a-d)/2$, $z= (b+c)/2$ and $t= (b-c)/2$. 
These elements are eigenvectors for the action of~$G$, and we have
\[
A = k[x,y,z,t]/(x^2 - y^2 - z^2 + t^2 - 1) \, .
\]
Using Theorem~\ref{thm-presentation}, it is an exercise to show the following.

\begin{prop}\label{prop-SL2}
The algebra~$SL_F(2)$ has a presentation with four generators $X$, $Y$, $Z$, $T$ 
subject to the seven relations
\[
XY = YX\, , \quad XZ = ZX \, , \quad XT = TX \, , 
\]
\[
YZ  = - YZ \, ,  \quad YT = -TY \, , \quad ZT = -TZ \, , 
\]
\[
X^2 + Y^2 + Z^2 - T^2 = 1 \, . 
\]
\end{prop}

As an illustration, we may take $k$ to be the field~$\R$ of real numbers
and describe the set ${\Alg}(SL_F(2), \R)$ of real points, 
which we denote by $SL_F(2, \R)$.

\begin{coro}\label{coro-sl2R}
The set $SL_F(2, \R)$ consists of two circles and a hyperbola, all intersecting in two points.
\end{coro}

\begin{proof}
An element of~$SL_F(2, \R)$ can be described as a point~$(x,y,z,t) \in \R^4$ 
satisfying $yz=0$, $yt=0$, $zt=0$ and $x^2+y^2 + z^2 - t^2 = 1$. 
If $y \ne 0$, then the point lies on the circle $x^2 + y^2 =1 $ in the plane $z=t=0$; 
if $z \ne 0$, then the point is on the circle $x^2 +  z^2 = 1$ in the plane $y=t=0$; 
if $t\ne 0$, then the point is on the hyperbola $x^2 - t^2 = 1$ in the plane $y=z=0$. 
Conversely, these three curves wholly lie in $SL_F(2, \R)$. 
Their intersection is reduced to the points~$(\pm 1,0,0,0)$.
\end{proof}

\begin{figure}[h]
\centering
\includegraphics[width=8cm]{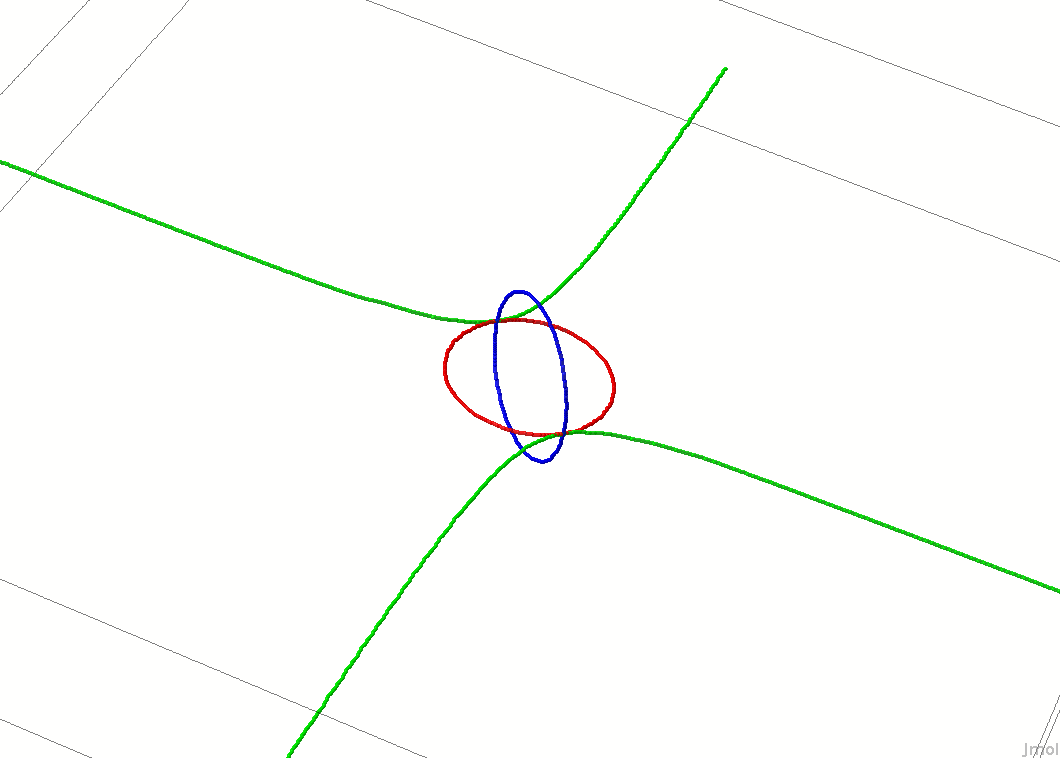}
\caption{\emph{An artist's view of}~$SL_F(2, \R)$}
\end{figure}

The algebra $SL(2)$ is a Hopf algebra with coproduct 
$\Delta : SL(2) \to SL(2) \otimes SL(2)$ given by
\[
\Delta(x) = xx' + yy' + zz' - tt' \, ,  \qquad \Delta(y) = xy' + yx' -zt' + tz'  \, , 
\]
\[
\Delta(z) = xz' + yt' + zx'  - ty' \, ,  \qquad \Delta(t) = xt' + yz' - zy'  + tx' \,  ,
\]
where $x$ is identified with $x\otimes 1$ and $x'$ with $1 \otimes x$
(similarly for the other variables).
By Proposition~\ref{prop-twist-hopf}, 
the twisted algebra~$SL_F(2)$ is a Hopf algebra in the braided sense.  
Let us denote $(SL(2) \otimes SL(2))_F$ by~$SL_F(2,2)$ for short.

\begin{prop}\label{prop-SL22}
The algebra $SL_F(2,2)$ is generated 
by eight generators $X$, $Y$, $Z$, $T$, $X'$, $Y'$, $Z'$, $T'$ 
subject to the following relations: 

\begin{itemize}
\item[(i)] 
the ``left relations'', which are as in Proposition~\ref{prop-SL2},

\item[(ii)] 
the ``right relations'', which are obtained from the left relations
by applying the substitutions 
$X \mapsto X'$, $Y\mapsto Y'$, $Z\mapsto Z'$, $T\mapsto T'$,

\item[(iii)] 
the ``composability conditions'', 
namely $X$ and $X'$ commute with all other generators, and 
\[ 
YZ' = - Z'Y\, , \quad YT'=-T'Y\, , \quad ZT'=-T'Z\, , 
\]
\[
Y'Z = - ZY'\, , \quad Y'T=-TY'\, , \quad Z'T=-TZ'\, .  
\]
\end{itemize}
\end{prop}

\begin{proof}
This is an exercise using Theorem~\ref{thm-presentation}. 
Alternatively, one can obtain the relations using the fact that $SL_F(2,2)$ is 
the braided tensor product of $SL_F(2)$ with itself.
\end{proof}

\begin{prop}\label{prop-Fdelta}
The map $\Delta_F : SL_F(2) \to SL_F(2,2)$ is given by the following formulas: 
\[
\Delta(X) = XX' - YY' - ZZ' + TT' \, ,  \qquad  \Delta(Y) = XY' + YX' - ZT' - TZ'  \, , 
\]
\[
\Delta(Z) = XZ'- YT'  +  ZX'  - TY' \, ,  \qquad \Delta(T) = XT' + YZ' + ZY' + TX'  \, .  
\]
\end{prop}

\begin{proof}
We know for example that $\Delta(X)  = xx' + yy' + zz' - tt'$. 
Using Lemma~\ref{lem-eigenvectors}, we see that~$XX' = xx'$, that~$YY' = -yy'$, and so on.
\end{proof}

When $R$ is a \emph{commutative} algebra, then $SL(2, R) = {\Alg}(SL(2), R)$ is a group; 
for a general algebra~$R$ however, the set~$SL(2, R)$ has only a partially defined group law 
(essentially, one can only multiply two matrices if all the coordinates commute). 
A similar statement holds for~$SL_F(2)$: two points of $SL_F(2, R) = {\Alg}(SL_F(2), R)$
are composable if and only if they satisfy the composability conditions
of Proposition~\ref{prop-SL22}.

We are now in position to describe the partially defined group law on the set~$SL_F(2, \R)$
of real points of~$SL_F(2)$. 
Let $\mathcal{C}_1$, $\mathcal{C}_2$ denote the two circles 
and~$\mathcal{H}$ the hyperbola of Corollary~\ref{coro-sl2R}.

\begin{coro}
Two points of $SL_F(2, \R)$ can be composed if and only if they both belong to 
one of~$\mathcal{C}_1$, $\mathcal{C}_2$ or~$\mathcal{H}$.
The groups $\mathcal{C}_1$ and $\mathcal{C}_2$ are isomorphic to the
group of complex numbers of modulus~$1$, 
while $\mathcal{H}$ is isomorphic to the multiplicative group of non-zero real numbers.
\end{coro}

\begin{proof}
The conditions of composability for $(x,y,z,t)$ and $(x', y', z', t')$
in this case are $yz'=0$, $yt'=0$, $zt'=0$, $zy'=0$, $ty'=0$ and $tz'=0$. 
The first statement follows from this.

From Proposition~\ref{prop-Fdelta} we deduce that 
the product of $(x, y,0, 0)$ and $(x', y', 0, 0)$ is $(xx' - yy', xy' + yx', 0, 0)$,
so that $\mathcal{C}_1$ is isomorphic to the group of complex numbers of
modulus~$1$ via the map $(x,y,0,0) \mapsto x + \sqrt{-1} \, y$.

Likewise, the product of $(x, 0,z, 0)$ and $(x',0, z', 0)$ is $(xx' - zz',0, xz' + zx', 0)$, 
which implies that $\mathcal{C}_2$ is isomorphic to~$\mathcal{C}_1$.

Finally, the product of $(x, 0, 0, t)$ and $(x', 0, 0, t')$ is $(xx' + tt', 0, 0, xt' + tx')$, 
so that we obtain an isomorphism $\mathcal{H} \to \R^\times$ with $(x, 0, 0, t) \mapsto x+t$ 
(recall that $(x+t)(x-t)=1$ on~$\mathcal{H}$).
\end{proof}

\section{A twisted enveloping algebra}\label{sec-u-sl2}

Up to this point, we have mostly explored those properties of $A_F$
which are $G$-equivariant in some sense or other. Typically these
properties are inferred from those of $A$, via the pair of inverse
functors~$( - )_F$ and~$( - )_{F^{-1}}$. 

If one forgets the $G$-action, and looks at $A_F$ simply as an algebra, 
genuinely new phenomena appear. For example, we may wish to
study all $A_F$-modules, as opposed to $G$-$A_F$-modules. 

In this section we do precisely this when $A$ is the universal enveloping algebra 
of the Lie algebra~$\slt$. It turns out that the simple $A_F$-modules are considerably 
more involved than the $A$-modules. 

\subsection{Presentation}

We now turn to the universal enveloping algebra $A = U(\slt)$ 
of the Lie algebra of traceless $2 \times 2$-matrices. 
To simplify matters, we assume that the ground field is the field~$\C$ of complex numbers.
We have
\[
U(\slt) = \C \langle a, b, h  \rangle / \left( [a,b]= -2h\, , \; [a,h]= -2b\, , \; [b,h] = -2a \right) \, ,   
\]
where 
\[
a=\left(\begin{array}{rr}
0 & 1 \\
1 & 0
\end{array}\right) \, , \qquad
b= \left(\begin{array}{rr}
0 & 1 \\
-1 & 0
\end{array}\right) \, , \qquad
h= \left(\begin{array}{rr}
1 & 0 \\
0 & -1
\end{array}\right) \, . 
\]
(These are slightly different from the standard generators.)

The group $G = \Z/2 \times \Z/2$ acts on $A$ essentially as it did in
Section~\ref{subsec-func-on-sl2}, that is 
\[
e_1 \cdot a =a \, , \;\; 
e_2 \cdot a = -a \, , \;\;
e_1 \cdot b = e_2 \cdot b = -b \, , \;\;
e_1 \cdot h = -h \, , \;\;
e_2\cdot h = h\, . 
\]
We consider the same twist~$F$ as before, and we write $U_F(\slt)$ for~$A_F$.
It follows from Proposition~\ref{prop-twist-hopf} that $U_F(\slt)$ is a Hopf algebra
in the braided category~$\mathcal{M}^V_b$, where $b$~is the unique non-trivial
alternating bicharacter on~$\widehat{V}$.
Let us give a presentation of~$U_F(\slt)$.

\begin{prop}\label{prop-sl2F}
The algebra $U_F(\slt)$ is generated by three non-commutative
variables $A$, $B$, $H$ subject to the relations
\[
AB + BA = -2H \, , \qquad AH + HA = 2 B \, , \qquad BH + HB = -2 A \, . 
\]
Alternatively, it is 
generated by three non-commutative variables $X, Y, Z$ subject to the relations
\begin{equation}\label{relationsFU}
XZ + ZX = 2X \, , \qquad YZ + ZY = -2Y \, , \qquad X^2 - Y^2 = Z \, . 
\end{equation}
\end{prop}

\begin{proof}
The relations involving $A, B$ and $H$ follow directly from Theorem~\ref{thm-presentation}. 
The other generators are related to $A,B,H$ by 
\begin{equation*}
X= \frac{1} {2} (B + \sqrt{-1} \, A) \, , \quad Y= \frac{1} {2} (B - \sqrt{-1} \, A) \, , \quad Z = -\sqrt{-1} \, H \, . 
\end{equation*}
One easily checks Relations~\eqref{relationsFU}.
\end{proof}

Since $U(\slt)$ carries a Hopf algebra structure, there is a twisted version of
the comultiplication $\Delta : U(\slt) \to U(\slt \times \slt )$. Let us
write $U_F^{(2)}(\slt)$ for $U(\slt \times \slt )_F$.

\begin{prop}
The algebra $U_F^{(2)}(\slt)$ is generated by six non-commutative
variables $X, Y, Z, X', Y', Z'$ subject to the relations~\eqref{relationsFU}
involving the triple~$(X,Y,Z)$, the same relations in which one replaces
$(X,Y,Z)$ by~$(X',Y',Z')$, as well as the relations
\[
XX' = Y'Y \, , \quad X'X= YY' \, , \quad XY' = Y'X \, , \quad YX' = X'Y \, ,  \quad ZZ' = Z'Z \, ,
\]
\[
XZ' = -Z'X \, , \quad YZ'= -Z'Y \, , \quad X'Z = -ZX' \, , \quad Y'Z = - ZY' \, .
\]
The diagonal map 
$\Delta_F : U_F(\slt) \to U_F^{(2)}(\slt)$
is given by 
\[
\Delta_F(X) = X + X' \, , \quad \Delta_F(Y) = Y + Y' \, , \quad \Delta_F(Z) = Z + Z' \, .
\]
\end{prop}

\subsection{Five families of simple modules} \label{subsec-examples}

From now on we write $U$ for~$U_F(\slt)$.
We shall describe five families of simple $U$-modules, distributed in two classes, 
called even and odd respectively. 

\subsubsection{Even modules.} 

For each integer $n \geq 0$, we describe a simple $U$-module~$E_n$.  
It has dimension $2n+1$ with a basis comprised of vectors~$v_k$ for $0\le k \le n$ 
and vectors~$w_k$ for $0 \le k < n$. The action of~$Z$ is given by 
\[
Zv_k = (2n- 4k) v_k \, , \qquad Zw_k = (4k+2 - 2n) w_k \, .  
\]
Thus all the eigenvalues of $Z$ are even integers, which is why
we call the modules of the form~$E_n$ even modules.

The action of $X$ is given by
\[
Xv_k = w_k \, , \qquad Xw_k = (2n + k(4n-2-4k))v_k \, ,  
\]
with the convention $w_n = 0$. 
As for the action of~$Y$, we have 
\[
Yv_k = k(4n+2-4k) w_{k-1} \, , \qquad Yw_k =  v_{k+1} \, ,  
\]
with the convention $w_{-1} = 0$. 
The reader is invited to check that the operators $X,Y,Z$ satisfy Relations~\eqref{relationsFU} 
and that the $U$-modules~$E_n$ are all simple.

The one-dimensional module~$E_0$ is the trivial module with $X = Y = Z = 0$.

\begin{ex}
In the basis $(v_0, w_1, v_1, w_0, v_2)$ the action of $X,Y,Z$ on~$E_2$ is
given by
\[
X= \left(\begin{array}{ccccc}
0 & 0 & 0 & 4 & 0 \\
0 & 0 & 1 & 0 & 0 \\
0 & 6 & 0 & 0 & 0 \\
1 & 0 & 0 & 0 & 0 \\
0 & 0 & 0 & 0 & 0
\end{array}\right),
\qquad
Y= \left(\begin{array}{ccccc}
0&0&0&0&0\\
0&0&0&0&4\\
0&0&0&1&0\\
0&0&6&0&0\\
0&1&0&0&0
\end{array}\right),
\]
and 
\[
Z= \left(\begin{array}{ccccc}
4&0&0&0&0\\
0&2&0&0&0\\
0&0&0&0&0\\
0&0&0&-2&0\\
0&0&0&0&-4
\end{array}\right).
\]
\end{ex}

\subsubsection{Odd modules.} 

There are four ``odd modules'' corresponding to an integer~$n$.
We call them $A_n^+$, $A_n^-$, $B_n^+$, and~$B_n^-$.

The module $A_n^\pm$ has dimension $n$, with a basis comprised of 
vectors~$v_k$ for $0\le k < n$, and the action of $Z$ on it is given by
\[
Zv_k = (2n-1-4k) v_k \, ,   
\]
so that all the eigenvalues of $Z$ are odd integers. For convenience
we set $w_k = v_{n-1-k}$ (so that symmetrically, $v_k = w_{n-1-k}$).

As for the actions of $X$ and~$Y$, we need to make a distinction 
according as $n$ is even or not. Consider the equations
\begin{equation}\label{eq-X-on-An}
Xv_k =  w_k \, , \qquad Xw_k=(2n-1 + 4k(n-k-1))v_k \, ,
\end{equation}
and
\begin{equation}\label{eq-Y-on-An}
Yw_j = v_{j+1} \, , \qquad Yv_k = 4k(n-k) w_{k-1} \, .
\end{equation}

If $n$ is even, we set $n=2m$.  
The action of $X$ is given by~\eqref{eq-X-on-An} for $0 \le k \le m-1$. 
The action of $Y$ is given by~\eqref{eq-Y-on-An} for $0 \le k \le m-1$ and 
$0 \le j \le m-2$, with the convention $w_{-1}=0$, and by 
\[
Yw_{m-1} = Yv_m = \pm n \, v_m \, .
\]

If $n$ is odd, we set $n=2m+1$. 
The action of $X$ satisfies~\eqref{eq-X-on-An} for $0 \le k \le m-1$, and we have 
\[
Xv_m = Xw_m = \pm n \, v_m \, .
\] 
The action of $Y$ is given by~\eqref{eq-Y-on-An}, 
this time with $0 \le j \le m-1$ and $0 \le k \le m$.

\begin{ex}
The module structure on~$A_2^{\pm}$ is given by
\[
X= \left(\begin{array}{cc}
0&3\\
1&0
\end{array}\right), \quad 
Y= \left(\begin{array}{cc}
0&0\\
0&\pm 2
\end{array}\right), \quad
Z= \left(\begin{array}{cc}
3&0\\
0&-1
\end{array}\right).
\]
\end{ex}

Let us first describe the one-dimensional module $B_0^\pm$:
one has $X=0$, $Y=\pm 1$, and $Z= -1$. Now assume $n> 0$.

The module~$B_n^\pm$ is, in a way, an ``augmented'' version of~$A_n^\pm$. 
It has dimension~$n+1$, with basis vectors~$v_k$ with $0\le k < n$ as above, 
together with an extra element~$u$.  
We keep the notation $w_k = v_{n-1 -k}$. We shall see that~$u$ can be thought of
as playing the role of~$v_n$ or~$w_{-1}$ (in other words, according to
their taste the readers can prefer to understand the formulas below
involving~$v_k$ or~$w_k$ beyond their scope with the conventions
$u=v_n=w_{-1}$ and $w_n= v_{-1}=0$; or they may prefer the explicitly
given formulas involving~$u$).

The action of $Z$ is given by
\[
Zv_k = (2n-1-4k) \, v_k \, , \qquad Zu = (-2n-1) \, u \, ,  
\]
so that the eigenvalues are again odd integers. We always have 
\[
Xu=0 \, , \qquad  Yv_0 = (2n+1) \, u \, , \qquad Yu= v_0 \, .
\]

Again we need to distinguish between $n$ even and $n$ odd. 
Here the equations to consider are
\begin{equation}\label{eq-X-on-Bn}
Xv_k =  w_k \, , \qquad Xw_k=(4n + 4k(n-k-1)) \, v_k  \, ,  
\end{equation}
and
\begin{equation}\label{eq-Y-on-Bn}
Yw_j = v_{j+1} \, , \qquad Yv_k = (2n+1 + 4k(n-k)) \, w_{k-1} \, .
\end{equation}

If $n= 2m$, the action of~$X$ is given by~\eqref{eq-X-on-Bn} for $0 \le k \le m-1$. 
That of~$Y$ is given by~\eqref{eq-Y-on-Bn} for $0 \le k \le m-1$ and $0 \le j \le m-2$, 
while 
\[
Yw_{m-1} = Yv_m = \pm (n+1)\,  v_m \, .
\]

For $n= 2m + 1$, the action of $X$ satisfies~\eqref{eq-X-on-Bn} for $0 \le k \le m-1$ again, 
and we also have 
\[
Xv_m = Xw_m = \pm (n+1)\, v_m \, .
\]
The action of~$Y$ is given by~\eqref{eq-Y-on-Bn} for $0 \le k \le m$ and $0 \le j \le m-1$.

\begin{ex}
For $B_1^\pm$ we have
\[
X= \left(\begin{array}{cc}
\pm 2&0\\
0&0
\end{array}\right), \quad
Y= \left(\begin{array}{cc}
0&3\\
1&0
\end{array}\right), \quad
Z= \left(\begin{array}{cc}
1&0\\
0&-3
\end{array}\right).
\]
\end{ex}

\begin{ex}
For $B_2^\pm$,
\[
X= \left(\begin{array}{ccc}
0&8&0\\
1&0&0\\
0&0&0
\end{array}\right), \quad
Y= \left(\begin{array}{ccc}
0&0&5\\
0&\pm 3&0\\
1&0&0
\end{array}\right), \quad
Z= \left(\begin{array}{ccc}
3&0&0\\
0&-1&0\\
0&0&-5
\end{array}\right).
\]
\end{ex}

It is easily checked that the modules $A_n^+$, $A_n^-$, $B_n^+$, and $B_n^-$ are all simple.
Note that there are precisely five modules of dimension~one, 
namely $A_1^+$, $A_1^-$, $B_0^+$, $B_0^-$, and the trivial module~$E_0$.

\subsection{The classification}

\begin{thm}\label{classification}
Any finite-dimensional non-zero simple $U$-module is isomorphic to one of 
the modules $E_n$, $A_n^+$, $A_n^-$, $B_n^+$, or $B_n^-$.
\end{thm}

In other words, we have described all the finite-dimensional simple
$U$-modules in Section~\ref{subsec-examples}.

We need the following proposition for the proof of the theorem.

\begin{prop}\label{prop-omnibus}
Let $V$ be a non-zero simple $U$-module. 

\begin{enumerate}
\item[(a)]
Assume that $Zv = \lambda v$ for some $\lambda \in k$ and $v\in V$. Then 
\[
Z( Xv ) = (2- \lambda ) Xv \quad\text{and}\quad
Z(Yv) = (-2 - \lambda ) Yv \, .
\]
If moreover $X^2 v = av$ for some scalar~$a$, then~$X^2(Yv)=(a - 2 \lambda -2) Yv$.

\item[(b)] 
Let $v$ be an eigenvector for both $Z$ and $X^2$, and let~$T=XY$. 
Then $V$ has a basis consisting of elements of the form $(Y^e T^n X^f) v$, 
where $n\ge 0$, while $e$ and $f$ are each equal to~$0$ or~$1$. 
In particular, $X^2$, $Y^2$ and $Z$ can be diagonalized simultaneously. 

\item[(c)] 
The eigenvalues of $Z$ have all multiplicity~one.

\item[(d)] 
Let $\lambda_0$ be the largest eigenvalue of~$Z$ (lexicographically), 
and let $v$ be a corresponding eigenvector. 
If $\dim V \ge~3$, then $Xv \ne 0$.
\end{enumerate}
\end{prop}

In (d) we order the complex numbers lexicographically, that is, we
identify $\C$ with $\R \times \R$, and set $(x,y) < (x', y')$ if and
only if $x < x'$, or $x=x'$ with $y < y'$. All the complex numbers
involved will turn out to be integers in the sequel!

\begin{proof}
(a) This follows from the relations in Proposition~\ref{prop-sl2F}.

(b) If $v$ is any non-zero vector, 
then $V$ has a basis consisting of elements of the form~$m(X,Y) \, v$, 
where $m(X,Y)$ is a monomial in~$X$ and~$Y$. 
Now, we can certainly find a vector~$v$ that is an eigenvector both for~$Z$ and~$X^2$, 
since by~(a) the operator~$X^2$ preserves the eigenspaces of~$Z$.
Another application of~(a) shows that $m(X,Y) \, v$ is again an
eigenvector for~$Z$ and~$X^2$, and so also for~$Y^2$ 
since $X^2 - Y^2 = Z$; this proves the last statement in~(b). 
Moreover, all occurrences of~$X^2$ or~$Y^2$ can clearly be removed from~$m(X, Y)$, 
so $V$ has a basis as announced.

(c) Let $\lambda $ be an eigenvalue of $Z$, and let $v$ be as in~(b),
satisfying $Zv = \lambda v$ and~$X^2v=av$. 
If a vector (other than~$v$ itself) is of the form $(Y^e T^n X^f) \, v$ 
and is an eigenvector for~$Z$ with eigenvalue~$\lambda $, 
then it must be either~$YT^n v$ in which case $\lambda = -1 -2n$, 
or~$T^n Xv$ in which case $\lambda = 2n+1$
(this is a simple verification from~(a)). 
Since $\lambda $ is either positive or negative, these possibilities are mutually exclusive: 
let us finish the proof in the case $\lambda = -1 -2n$, leaving the other case to the reader. 
We must prove that $w= YT^n v$ is a multiple of~$v$. Assume it is not. 
A straightforward but lengthy calculation shows that $X^2 w = a w$. 
Moreover, $YT^n w$ is a multiple of~$v$, say $YT^n w = \beta v$, as one sees readily. 
Choose $\alpha $ such that $\beta = \alpha^2$, and put $v' = \alpha v + w$. 
Now run through the proof with~$v'$ replacing~$v$. This time we do have
\[
YT^n v' = \alpha w + \beta v = \alpha v' \, .  
\]
So the eigenspace for the eigenvalue~$\lambda$ is one-dimensional. 

(d) Assume that $Xv = 0$, and let $w= Yv$. 
One has $Zw = -(\lambda_0 + 2) w$ by~(a); 
on the other hand $Xw = 0$, for otherwise by~(a) it
would be an eigenvector associated to $\lambda_0 + 4 > \lambda_0$; 
and finally $Yw$ is a multiple of~$v$ by~(c) 
since it satisfies $Z(Yw) = \lambda_0 Yw$. 
It follows that $v$ and $w$ generate a $U$-submodule of dimension~$\le 2$,
which is equal to~$V$ in view of the simplicity of the latter.
\end{proof}

As preliminaries for the proof of Theorem~\ref{classification}, we
need some notation. 
Let $V$ be a simple $U$-module. 
For any complex number~$\lambda$, we define a vector $e_\lambda \in V$ as follows: 
if $\lambda $ is an eigenvalue for~$Z$, we pick an associated eigenvector~$e_\lambda $; 
otherwise, we set $e_\lambda = 0$. 
By Proposition~\ref{prop-omnibus}\,(c) each vector~$e_\lambda $ is uniquely defined up to a scalar.

Next, we define complex numbers~$X(\lambda)$ and~$Y(\lambda )$ as follows. 
By Propo\-si\-tion~\ref{prop-omnibus}\,(a), the vector $X e_\lambda $ is either~$0$, 
in which case we set $X(\lambda ) =0$, or a multiple of $e_{2-\lambda }$, 
in which case we choose $X(\lambda)$ so that
\[
X e_\lambda = X( \lambda )\, e_{2- \lambda } \, .  
\]
Note that this relation is then true for all $\lambda \in  \C$. Similarly, we define $Y( \lambda )$ so that
\[
Y e_\lambda = Y( \lambda )  \, e_{-2-\lambda } \, . 
\]

From the relation $X^2 - Y^2 = Z$, we easily draw the following lemma.

\begin{lem}\label{lem-base}
If $e_\lambda \ne 0$, then
\[
X( \lambda ) \, X( 2 - \lambda ) - Y( \lambda ) \, Y(-2-\lambda ) = \lambda \, .   
\]
\end{lem}

Let $\lambda_0$ be the highest eigenvalue of~$Z$ (lexicographically). 
The central idea is to pay attention to the complex numbers $c_n =
Y(4n-2-\lambda_0) \, Y( \lambda_0 - 4n)$ (they will turn out to be
integers). 

\begin{lem}\label{lem-recurrence}
Let $\lambda = \lambda_0 - 4n$ for some integer $n\ge 0$. If

\begin{itemize}
\item[(i)]
$e_\lambda \ne 0$ and 

\item[(ii)]
$e_{2-\lambda } \ne 0$, 
\end{itemize}
then $c_{n+1} = c_n + 2 \lambda_0 - 8n -2$.
\end{lem}

From Proposition~\ref{prop-omnibus}\,(c) we see that Lemma~\ref{lem-recurrence} 
applies at least once, namely for~$n=0$. 
From now on, the number~$n$ is chosen to be the smallest integer such that 
this lemma cannot be applied, because either~(i) or~(ii) does not hold. 
Clearly, $n$ is finite, as $V$ is finite-dimensional.

By an immediate induction, we see that for $0 \le k \le n$,
\begin{equation}\label{eq-ck}
c_k = c_1 + 2(k-1)(\lambda_0 -2k-1) \, .
\end{equation}

It is in fact easy to compute~$c_1$: there are two cases to consider.

\smallskip
\emph{Case (1).} 
Assume $Y( \lambda_0 ) = 0$. From Lemma~\ref{lem-base} with $\lambda = \lambda_0$, 
we have $X( \lambda_0 ) \, X( 2- \lambda_0 ) = \lambda_0$. 
Applying the same lemma with $\lambda = 2 - \lambda_0$
(which we may by Proposition~\ref{prop-omnibus}\,(c)), 
we obtain $c_1 = 2 (\lambda_0 - 1)$. It follows that
\begin{equation}\label{eq-cn-case1}
 c_n = 2(n \lambda_0 - (n-1)(2n+1) -1) \, . 
\end{equation}

\emph{Case (2).} Assuming $Y( \lambda_0 ) \ne 0$, we can apply
Lemma~\ref{lem-base} with $\lambda = -2 - \lambda_0$. 
Keeping in mind that $X( \lambda_0 + 4 ) = 0$ by maximality of~$\lambda_0$, 
we draw $c_0 = 2 + \lambda_0$, so that $c_1 = 3 \lambda_0$ and 
\begin{equation}\label{eq-cn-case2}
c_n = (2n+1) \lambda_0 -2(n-1)(2n+1) \, . 
\end{equation}
This completes the preliminaries. 

\pagebreak[3]

\medskip\noindent
\emph{Proof of Theorem~\ref{classification}}.
It is divided in two cases, according to which 
hypothesis of Lemma~\ref{lem-recurrence} fails to hold for~$n$ 
(see above the definition of this integer).

\smallskip
\emph{Case (a): Hypothesis}\,(i) \emph{holds, but} (ii) \emph{does not.} 
In other words, $e_\lambda \ne 0$, but $e_{2-\lambda } =0$ for $\lambda = \lambda_0 - 4n$. 
Thus we certainly have $X(2- \lambda ) =0$, and Lemma~\ref{lem-base} for this $\lambda $ 
gives $c_n = - \lambda = 4n - \lambda_0$. 
From this we deduce the value of~$\lambda_0$ as follows.

\begin{itemize}

\item 
\emph{Subcase (a) \& (1).} 
Combining~\eqref{eq-cn-case1} with the equation $c_n = 4n - \lambda_0$ just obtained, 
we see that $\lambda_0 = 2n$. 
Let us prove that $V$ is isomorphic to~$E_n$ in this case. 
Let $\lambda_k = \lambda_0 - 4k$, and let $v_0 =  e_{\lambda_0}$. 
When $0 \le k < n$, we can apply Lemma~\ref{lem-base} with $\lambda = \lambda_k$, 
and obtain
\[
X( \lambda_k ) \, X( 2 - \lambda_k ) = c_k + \lambda_k \ne 0 \, .  
\]
The fact that $c_k + \lambda_k \ne 0$ follows from~\eqref{eq-ck}, and 
the same equation shows that $c_k \ne 0$ for these values of~$k$.

In particular, we have $X( \lambda_k ) \ne 0$ and $Y( \lambda_k ) \ne 0$ for $0 \le k < n$. 
Thus we may set $w_0 = Xv_0$, $v_1 = Yw_0$, then
$w_1 = Xv_1$, $v_2=Yw_1$, and so on until $w_{n-1}= X v_{n-1}$, $v_n = Yw_{n-1}$; 
all these vectors are non-zero. 
By Proposition~\ref{prop-omnibus}\,(a), they are eigenvectors for different eigenvalues of~$Z$, 
and as such they are linearly independent.

It is a consequence of the definitions that
\[
Zv_k = (2n- 4k) \, v_k \, , \quad Zw_k = (4k+2 - 2n) \, w_k \, ,   
\]
\[
Xv_k = w_k \, , \quad Xw_k = \alpha_k v_k \, ,   
\]
\[
Yv_k = \beta_k w_{k-1} \, , \quad Yw_k =  v_{k+1}  
\]
for $0 \le k \le n$ with the convention $w_n = w_{-1} = 0$: 
indeed, $Xv_n = 0$ since we are in Case\,(a), and $Y v_0 = 0$ 
since we are in Case\,(1).

As a result, the elements $v_k$ and $w_k$ generate a $U$-submodule of~$V$, 
hence all of~$V$ since the latter is simple. It remains only to compute the
value of the scalars~$\alpha_k$ and~$\beta_k$.

We may assume that $\beta_0=0$. Checking the relation $X^2 - Y^2 = Z$
against~$v_i$ gives $\alpha_0 = 2n$ for $i=0$, and 
\[
\alpha_i - \beta_i = 2n-4i \quad\textnormal{for}\; i \ge 1 \, .   
\]
Checking the same relation against $w_i$ yields 
\[
\alpha_i - \beta_{i+1} = 4i +2 -2n \quad\textnormal{for}\; i \ge 0 \, .  
\]
Comparing the two equations gives $\beta_{i+1}- \beta_i = 4n-8i-2$ for $i\ge 1$; 
now sum this for $i$ between~$1$ and~$k-1$, use $\beta_1 = 4n-2$, 
and obtain $\beta_k = k(4n+2-4k)$. Then solve for~$\alpha_k$.

\item 
\emph{Subcase (a) \& (2).} This is left to the reader. 
One finds $\lambda_0 = 2n - 1$, and $V$ is isomorphic to~$B_n^\pm$.

\end{itemize}

\emph{Case\,(b): Hypothesis}\,(i) \emph{does not hold.} 
In other words, $e_{\lambda_0 -4n}=0$. Thus,
\[
Y( \lambda_0 - 4n ) = 0 = c_n \, .
\] 
One can easily solve this for~$\lambda_0$: 

\begin{itemize}

\item \emph{Subcase (b) \& (1).}  One finds $\lambda_0 = 2n - 1$. It
  follows that $V$ is isomorphic to~$A_n^\pm$ (details are left to the
  reader).

\item \emph{Subcase (b) \& (2).} 
Let us show that this case does not occur. 
Indeed, from~\eqref{eq-cn-case2} we see that $\lambda_0 = 2n - 2$. 
On the other hand, by definition of~$n$, and since we are in Case\,(b), 
we know that $e_{2 - \lambda}\ne 0$ for $\lambda = \lambda_0 -4(n-1) = -2n+2$. 
This produces an eigenvector for~$Z$ with eigenvalue $2 - \lambda = 2n$,
 which contradicts the maximality of $\lambda_0$.
\hfill\qed
\end{itemize}

\begin{Rem}
It can be proved that the category of finite-dimensional $U_F(\slt)$-modules is not semisimple 
though that of~$U(\slt)$ is.
\end{Rem}

\appendix

\section{Versal extensions for Galois objects}

We apply the techniques of the paper to discuss an issue raised by
Eli Aljadeff and the second-named author in~\cite{AK}. 

The issue is one of rationality.  Let $k$ be a base field, $H$ a Hopf
algebra over~$k$, and let~$A$ be a cleft, left Galois object of~$H$
(recall the definition from
Section~\ref{subsec-two-cocycles-on-Hopf-algebras}).  If~$K$ is any
field containing~$k$, a \emph{form of~$A$ over~$K$} is a cleft
Hopf-Galois extension~$A'$ of~$K$, with structure Hopf
algebra~$H\otimes_k K$, such that $A'\otimes_K \bar K$ is isomorphic
to $A \otimes_k \bar K$, where~$\bar K$ is the algebraic closure
of~$K$.  The question essentially is: can one classify the forms
of~$A$ over an arbitrary field~$K$?  In fact we shall refine this
question below.

Before we do that, let us make a few comments. By definition, we know
that $A$ (resp.\ $A'$) is of the form ${}_\sigma H$ (resp.\ ${}_\tau
H$).  The condition for~$A'$ to be a form of~$A$ is precisely that
$\tau$ should be cohomologous to~$\sigma $ over~$\bar K$.  Since there
are non-trivial two-cocycles even over algebraically closed fields,
this condition is not vacuous. What is more, there are many forms
of~$A$ even if $\sigma $ is the trivial two-cocycle.

Let us recall the ``classical case''.  Say $H= \OO_k(G)$ for an
algebraic group $G$, and restrict attention to \emph{commutative}
algebras~$A$; in this situation, a cleft Galois object of~$H$ which is
also commutative is precisely the algebra of functions on a $G$-torsor
(this is well known, and the proof would take us too far afield).
Over an algebraically closed field there is only the trivial
$G$-torsor, and thus the only forms to study are those of the trivial
torsor.

In algebraic geometry, a very useful device for studying torsors is
afforded by \emph{versal extensions}. 
Such an extension is by definition a map of varieties $U\to B$, 
where $G$ acts freely on~$U$ and $B= U/G$, 
with the following two properties:

\begin{itemize}
\item[(i)] 
any $G$-torsor~$T$ is obtained as a fibre of this map along a map $\Spec(K) \to B$;

\item[(ii)]
any such fibre is a $G$-torsor.
\end{itemize}

When the varieties are affine, we may let~$\U$ (resp.\ $\B$) be the algebra of functions on~$U$
(resp.\ on~$B$), and we see that any $G$-torsor~$T$ is of the form
$\Spec(A)$ for $A = \U \otimes_\B K$; conversely, any such~$A$ defines a torsor.

It is an easy theorem that versal extensions always exist, and
for a finite group~$G$ we may construct such objects as follows: 
start with an embedding $\rho : G \to GL(V)$, and define~$U$ to be the complement in~$V$ 
of the subspaces $\ker( \rho( g ) - \id)$ for $g\in G$. 
The map $U \to B=U/G$ is a versal extension (see~\cite[I, Sect.~5]{GMS}). 
We can even arrange for~$U$ to be affine if needed: 
take a finite collection $(P_i)_i$ of hyperplanes in~$V$ such that each $\ker( \rho( g ) - \id)$ 
is contained in some~$P_i$, and such that $S =\cup_i\, P_i$ is $G$-invariant; 
then the complement $U= V - S$ is affine and $U \to B= U/G$ is again versal.

\begin{thm}
Let $H = \OO_k(G)$ for some finite group~$G$ and some field~$k$.  
Let~$\sigma $ be a lazy two-cocycle for~$H$.
If $\B \subset \U$ is a classical versal extension as above, then
${}_\sigma \B \subset {}_\sigma \U$ is a versal extension for forms of~${}_\sigma H$.
\end{thm}

We caution that~${}_\sigma H$ is an~$(H,H)$-bicomodule, and the
theorem refers to the underlying \emph{left} comodule. Also note that,
since the $G$-action on~$\B$ is trivial, there is a natural
identification of~${}_\sigma \B$ with~$\B$.

\begin{proof}
Let $K$ be a field containing~$k$ and let~$A$ be a form of~${}_\sigma H$ over~$K$.  
The~$H$-comodule algebra~${}_{\sigma^{-1}} A$ is a form of~$H$, 
and as such it is a Hopf-Galois extension of~$K$. 
The cleft condition is automatically satisfied by finite-dimensionality of the Hopf algebra~$H$.  
By definition of the classical versal extension, ${}_{\sigma ^{-1}} A$ is 
of the form~$\U \otimes_\B K$ for some algebra morphism $\eta : \B \to K$.  
As a result, $A$ is isomorphic to~${}_\sigma ( \U \otimes_\B K )$.

To complete the proof, it thus suffices to establish that ${}_\sigma( \U \otimes_\B K)$ 
can be identified with ${}_\sigma \U \otimes_{\B} K$ 
(keeping in mind that ${}_\sigma \B = \B$, ${}_\sigma K = K$ and ${}_\sigma \eta = \eta$). 
This is an exercise.
Moreover, ${}_\sigma \U \otimes_{\B} K$ is now clearly seen to be a
form of~${}_\sigma H$ for any algebra morphism $\eta: \B\to K$.
\end{proof}

\begin{Rem}
Aljadeff and the second-named author~\cite{AK} attached to any Hopf algebra~$H$ 
and any two-cocycle~$\sigma$ an $H$-Galois extension $\B_{H}^{\sigma} \subset \AA_{H}^{\sigma}$.
It turns out that this extension is a versal extension classifying all forms of~${}_{\sigma} H$
for a large class of Hopf algebras~$H$. 
By~\cite[Th.~3.6]{KM}, this class includes all finite-dimensional Hopf algebras.

When $H = \OO_k(G)$ for a finite group~$G$ and $\sigma = \eps$ is the trivial cocycle, then 
by~\cite{AK, KM}, the versal extension $\B_{H}^{\eps} \subset \AA_{H}^{\eps}$ has the following
description:
\[
\AA_{H}^{\eps} = k[\, t_g \, |\, g\in G \, ] \Bigl[\frac{1}{\Theta_G} \Bigr] \, ,
\]
where $\Theta_G = \det(t_{gh^{-1}})_{g,h \in G}$ is \emph{Dedekind's group determinant},
and
\[
\B_{H}^{\eps} = (\AA_{H}^{\sigma})^G
= k[\, t_g \, |\, g\in G \, ]{\,}^G \Bigl[\frac{1}{\Theta_G^2} \Bigr] \, ,
\]
where $G$ acts on the indeterminates~$t_g$ by translation of the indices.
These algebras are closely related to the classical versal extension $\B \subset \U$ 
described above in the case one uses the regular representation of~$G$. 

If $\sigma$ is a general lazy two-cocycle of~$\OO_k(G)$, then $\B_{H}^{\sigma} = \B_{H}^{\eps}$
as before and $\AA_{H}^{\sigma}$ is obtained from~$\AA_{H}^{\eps}$ 
as stated in~\cite[Prop.~2.3]{KM}. 
It can be checked that the Galois extension $\B_{H}^{\sigma} \subset \AA_{H}^{\sigma}$
is more or less the twisted versal extension ${}_\sigma  \B \subset {}_\sigma \U$ 
of the previous classical~$\B \subset \U$.
\end{Rem}

\section*{Acknowledgements}

The present joint work is part of the project ANR-07-BLAN-0229
``Groupes quantiques~: techniques galoisiennes et d'int\'egration" funded
by Agence Nationale de la Recherche, France.
We thank the referee for useful comments.

\end{document}